\documentclass[12pt, reqno]{amsart}

\usepackage{amssymb, bbm, amsmath, fullpage, graphicx, tikz, caption}

\usetikzlibrary{positioning}
\usepackage[colorlinks, backref=page, citecolor=blue]{hyperref}
\usepackage[backrefs, alphabetic]{amsrefs}

\numberwithin{equation}{section}
\theoremstyle{plain}
\newtheorem{thm}{Theorem}[section]
\newtheorem{cor}{Corollary}[section]

\newtheorem{prop}{Proposition}[section]
 
\theoremstyle{definition}
\newtheorem{definition}{Definition}[section]

\newtheorem{remark}{Remark}[section]
\newtheorem{example}{Example}[section]

\newcommand{\qede}{\hfill $\diamond$}
\newcommand{\N}{\mathbb N}
\newcommand{\R}{\mathbb R}
\newcommand{\Z}{\mathbb Z}

\newcommand{\T}{\mathbb T}
\newcommand{\Q}{\mathbb Q}

\newcommand{\spn}{\mathrm{span}}

\newcommand{\conv}{\mathrm{conv}}
\newcommand{\dom}{\mathrm{dom}}
\newcommand{\core}{\mathrm{core}}

\begin{document}
\title{Convex analysis in groups and semigroups: a sampler} 

\dedicatory{This paper is dedicated to  R. Tyrell Rockafellar on the occasion of his eightieth  birthday}

\thanks{This work was funded in part by the Australian Research Council}


\author{Jonathan M. Borwein \and Ohad Giladi}
\address{Centre for Computer-assisted Research Mathematics and its Applications (CARMA), School of Mathematical and Physical Sciences, University of Newcastle, Callaghan, NSW 2308, Australia}
\email{jonathan.borwein@newcastle.edu.au, ohad.giladi@newcastle.edu.au}

\maketitle

\begin{abstract}
We define convexity canonically in the setting of monoids. We show that many classical results from  convex analysis hold for functions defined on  such groups and semigroups, rather than only vector spaces. Some examples and counter-examples are also discussed.
\end{abstract}


\section*{Part I: Basic convex analysis}

\section{Introduction}

The notion of convexity is  classical \cite{Rock97}, and heavily used in diverse contexts \cite[Chapter 1]{BV10}. While normally considered in the concrete setting of vector spaces --- either $\R^d$ or infinite dimensional --- it has often been examined in  very general  axiomatic form, see \cite{BHT82} and \cite{Vel93}.
 In the vector space case, $x$ is said to be a \emph{convex combination} of $x_1,\dots,x_n$ if there exist $\alpha_1,\dots,\alpha_n\in (0,1)$ such that
\begin{align}\label{conv vec space}
x = \sum_{i=1}^n\alpha_i x_i,~~ \sum_{i=1}^n\alpha_i=1.
\end{align}
If we assume for a moment that $\alpha_i$ is of the form $\alpha_i = \frac{m_i}{\sum_{i=1}^nm_i}$ where $m_1,\dots,m_n\in \N = \{1,2,\dots\}$, then~\eqref{conv vec space} becomes
\begin{align}\label{conv add group}
mx = \sum_{i=1}^nm_ix_i,~~ m = \sum_{i=1}^nm_i.
\end{align}
In~\eqref{conv vec space} we must be able to define $\alpha x$ for $\alpha \in \R$ and $x\in X$. More generally,~\eqref{conv vec space} can be used whenever $X$ is a module. On the other hand, in~\eqref{conv add group} we use \emph{only} the additive structure of $X$, i.e., we may assume that $X$ is merely an additive semigroup. (See Section~\ref{sec prelim} for the exact definitions.) Using~\eqref{conv add group}, we show how one can build a \emph{canonical} theory of convexity for additive groups and semigroups.  We refer the reader to~\cite{Mur03, Vel93} for more information on abstract convexity in all its manifestations. Some aspects of convex analysis in a more abstract setting have also been studied in~\cite{Ham05, JLMS07, LMS04}. Note that in~\cite{LMS04} for example, it is only required that a function is convex over geodesic curves (in this case, in the Heisenberg group). Thus, the various notions of convexity do not always coincide. See also Remark 1 in~\cite{LMS04}.

In a similar fashion to~\eqref{conv add group}, one can  define convex functions on additive groups and semigroups (again, see Section~\ref{sec prelim}). It is then natural to ask whether one may obtain useful analogues of known results for convex functions. It turns out that under only minimal assumptions on the underlying monoid or group, it is possible to reconstruct many classical results from the theory of convex functions such as Hahn-Banach type theorems, Fenchel duality, certain constrained optimisation results, and more.
We dedicate Section~\ref{sec examples} to exhibiting  concrete examples of groups and their convex sets and convex hulls. It turns out that even in  simple examples, the structure of convex sets is subtle  and can differ significantly from the structure of convex sets in vector spaces.

The rest of the paper is dedicated to generalising  classical  results of the theory of convexity to  more general settings. While many of the results presented here hold when we assume that the underlying space is a module (see Section~\ref{sec conv funs}), for the sake of concreteness we formulate most of the results for groups and semigroups. 
In Section~\ref{sec interpolate} we discuss the interpolation of subadditive and convex function. In short, the question (say, in the convex case) is: given two functions $f$ and $g$ with $g\le f$ and $f$, and $-g$ are convex, can we find an affine function $a$ such that $g \le a \le f$. Such questions were studied in~\cite{MO53} and  generalised in~\cite{Kau66}. We show that interpolation is possible for convex functions on semigroups which are semidivisible (see Section~\ref{sec conv funs}).

Part II of this paper (Sections~\ref{sec operators} and~\ref{sec optimize}) is dedicated to the study of convex operators between (semi)groups. We define some well known and widely used notions, such as directional derivatives and conjugate functions in the groups setting. In Section~\ref{sec operators}, we show that some of the best known results, such as the the max formula, sandwich theorems and Fenchel type duality theorems extend to this general setting.
Finally, in Section~\ref{sec optimize} we briefly discuss optimisation over groups before making some concluding remarks in Section~\ref{sec conc}.

\section{Convex basics}\label{sec prelim}

We define convex sets and functions and examine some basic properties.

\subsection{Convexity in algebraic structures}

A \emph{semiring} is  a commutative semigroup under addition and a semigroup under multiplication. A (left) \emph{semimodule} over a semiring is a commutative monoid (i.e., semigroup), satisfying all axioms of a module over a ring except the existence of an additive inverse.

\begin{definition}[Convex set in semimodule]\label{def convex set}
Assume that $X$ is a semimodule over a semiring $R$, and $A\subseteq X$. Let $r_1,\dots,r_n\in R\setminus\{0\}$, and $x_1,\dots,x_n\in A$. Assume that there exists $x\in X$ satisfying 
\begin{align*}
rx=\sum_{i=1}^nr_ix_i, ~~ r= \sum_{i=1}^nr_i.
\end{align*}
If $x\in A$ for every choice of $n\in \N$, $r_1,\dots,r_n\in R\setminus \{0\}$ and $x_1,\dots,x_n\in A$, then $A$ is said to be convex.
\end{definition}
Herein we  always assume that $\N=\{1,2,\dots\}$, i.e., all \emph{positive} integers. If $R$ is a ring, not just a semiring, then we assume  it is equipped with a \emph{compatible} partial order, i.e., that we have $r+r_1\le r+r_2$ whenever $r_1\le r_2$ and $r\cdot r_1\le r\cdot r_2$ whenever $r_1\le r_2$ and $r\ge 0$, and in Definition~\ref{def convex set}, we take only elements that are strictly positive. In particular, if $R$ is a field with a compatible partial order, $R_+$ is the collection of all positive elements, and $r_1,\dots,r_n\in R_+\setminus\{0\}$, then we have 
\begin{align*}
\sum_{i=1}^nr_i = r ~~ \Longrightarrow ~~ \sum_{i = 1}^n\frac{r_i}{r} = 1, ~~ rx=\sum_{i=1}^nr_ix_i ~~ \Longrightarrow ~~ x=\sum_{i=1}^n\frac{r_i}{r}x_i,
\end{align*}
which gives the standard definition of convexity (e.g., over $\R$ or $\mathbb Q$). As in vector spaces, we can also define convex cones. 

\begin{definition}[Convex cone in semimodule] 
A set $A\subseteq X$ is said to be a convex cone if in Definition \ref{def convex set} the assumption $\sum_{i=1}^nr_i= r$ is not imposed.
\end{definition}

Every commutative group is a module over the $\Z$. Herein, we will focus on additive groups and semigroups. By a \emph{monoid} we mean an additive semigroup with a unit. As noted in~\cite{Ham05}, a monoid with a nontrivial idempotent element cannot be embedded in a group. Clearly every monoid is a semimodule over the semiring $\Z_{+}$. Thus, the elements in Definition \ref{def convex set} are positive integers, denoted $m_j$ instead of $r_j$. 

For a general commutative group, one cannot always solve the equation
\begin{align}\label{eq can divide}
\left(\sum_{i=1}^nm_i\right)x = \sum_{i=1}^nm_ix_i.
\end{align}
Yet, equation~\eqref{eq can divide} is very useful in some cases. Thus, we recall the following.
\begin{definition}[Divisible group]
An additive group $X$ is said to be divisible if for every $n\in \N$, $nX=X$. Alternatively, $X$ is divisible if for every $y\in X$ and for every $n\in \N$, there exists $x\in X$ such that $nx = y$.
\end{definition}

\begin{definition}[Semidivisible group]\label{semi div def}
An additive group is said to be $p$-semidivisible is there exists $p\in \N$ prime such that $pX=X$, and $X$ is said to be semidivisible if it is $p$-semidivisible for some prime $p$.
\end{definition}

We can similarly define \emph{divisible} and \emph{semidivisible monoids}, as well as \emph{divisible} and \emph{semidivisible semimodules}. In particular, all divisible submodules and divisible submonoids are convex cones. A notion which is stronger than the above two is the following.

\begin{definition}[Uniquely divisible group]
An additive group $X$ is said to be uniquely divisible if for every $n\in \N$ and for every $y\in X$, there exists a unique $x\in X$ such that satisfies $nx=y$. Alternatively, $X$ is said to be uniquely divisible if it is divisible and for every $n\in\N$, the map $x\mapsto nx$ is an injective map.
\end{definition}

Similarly, we can consider the following notion.

\begin{definition}[Uniquely divisible monoid]
A monoid $X$ is said to be uniquely divisible if it is divisible and for every $n\in \N$, the map $x\mapsto nx$ is an injective map.
\end{definition}

Note that in monoids, singletons are convex if and only if the monoid is uniquely semidivisible, since we want $\sum_{i=1}^nm_ix = \left(\sum_{i=1}^nm_i\right)x$ to be the same as $\left(\sum_{i=1}^nm_i\right)y$ if and only if $x=y$. Divisibility and semidivisibility are important  for the structure theory of infinite abelian groups. See for example~\cite{Fuc70, Rob96}. We also refer the reader to~\cite{KTW11, Law10} for some more recent examples relating to divisible groups.

\begin{remark}
A subgroup of a divisible group need not be divisible, or even semidivisible. As a simple example, take $X=\R$ and $\Z\subseteq X$. \qede
\end{remark}

\begin{remark}[Divisibility in abelian groups]\label{rmk divisible}
It is known that every abelian group is a subgroup of a divisible group. Moreover, the quotient of a divisible group is again divisible, e.g., $\R/\Z$ and $\Q/\Z$. Also, the torsion subgroup $T_G$ (of all elements of finite order) is divisible and the quotient $G/T_G$ is a $\Q$-vector space.  Finally, the divisible groups are exactly the \emph{injective} abelian groups.
\qede
\end{remark}

\begin{remark}
If $X$ is $p$-semidivisible, i.e., $pX=X$ then for every $l\in \N$ we have $p^lX = p^{l-1}(pX) = p^{l-1}X = \dots = pX=X$. \qede
\end{remark}

\begin{remark}
Assume  $X=nX$ for some $n\in \N$. Write $n = p_1^{m_1}\cdots p_l^{m_l}$. Then  $X = p_1^{m_1}\cdots p_l^{m_l}X = p_1\left(p_1^{m_1-1}\cdots p_l^{m_l}\right)X\subseteq p_1X \subseteq X$ and so $X= p_1X$. Thus, for us the assumption that $p$ is prime in Definition~\ref{semi div def} plays no significant r\^ ole. \qede
\end{remark}

As mentioned above, convexity has an entirely axiomatic approach. We refer the reader to~\cite{Vel93} for more information about this rich topic. We will present only the basic definitions and the return to the more concrete case of convexity in algebraic structures. 

\begin{definition}[Convexity]\label{def convexity}
A collection $\mathcal C$ of subsets of a set $X$ is said to be a convexity (also an \emph{alignment}), if it contains the empty set and is closed under intersections and directed unions.
\end{definition}
It is straightforward to check the convex sets defined by Definition~\ref{def convex set} form a convexity. Given the Definition~\ref{def convexity}, we can also define the convex hull.
\begin{definition}[Convex hull]\label{def conv hull}
If $A\subseteq X$, define
\[\conv(A) = \bigcap_{\substack{A\subseteq B\\ B \text { convex}}}B.\]
\end{definition}
The convex hull is a \emph{closure operator}, i.e., it satisfies the following: $1.~A\subseteq B \Longrightarrow \conv(A)\subseteq \conv(B)$; $2.~ A\subseteq \conv(A)$; $3.~ \conv(\conv(A))=\conv(A)$; $4.~ \conv(\emptyset)=\emptyset$; $5.$ Closure under intersections and directed unions.

In the case of monoids, we have the following concrete result.

\begin{prop}[Convex hull in monoid]\label{hull monoid}
If $X$ is a monoid and $A\subseteq X$, the convex hull of $A$ is given by
\begin{align}\label{for hull}
\conv(A) = \left\{x\in X~\left| ~mx = \sum_{i=1}^nm_ix_i, ~ x_i\in A, ~m_i\in \N, ~m = \sum_{i=1}^nm_i\right.\right\}.
\end{align}
\end{prop}

\begin{proof}
Clearly the set on the right side of~\eqref{for hull} is convex and contains $A$. If $A\subseteq B$ and $B$ is convex, then $B$ contains the set on the right side of~\eqref{for hull}. 
\end{proof}

A map $T:X_1\to X_2$ between two monoids is said to be \emph{additive} if $T(x_1+x_2) = Tx_1+Tx_2$ for all $x_1,x_2\in X_1$. It is well known that a linear image of a convex set in a vector space is again convex. We establish a similar fact for additive bijections between monoids.

\begin{prop}[Convexity under additive bijection]\label{prop add image}
Assume that $X_1,X_2$ are monoids and $T:X_1\to X_2$ is an additive bijection. If $A\subseteq X_1$ is convex, then $TA\subseteq X_2$ is convex.
\end{prop}

\begin{proof}
Assume that $m,m_1,\dots,m_n\in \N$ and $y_1,\dots,y_n\in TA$, $y\in X_2$ are such that $my = \sum_{i=1}^nm_iy_i$, $m=\sum_{i=1}^nm_i$. Since $T$ is onto, there exists $x\in X_1$ such that $Tx=y$. Since $y_1,\dots,y_n\in TA$, there exist $x_1,\dots,x_n\in A$ such that $y_1=Tx_1,\dots,y_n=Tx_n$. Hence, we have $T(mx) = mTx = my =\sum_{i=1}^nm_iy_i = \sum_{i=1}^nm_iTx_i = T\left(\sum_{i=1}^nm_ix_i\right)$. Since $T$ is injective, we have $mx = \sum_{i=1}^nm_ix_i$. Since $x_1,\dots,x_n\in A$ and $A$ is convex, it follows that $x\in X$. Thus, $y=Tx \in TA$ and $TA$ is convex. 
\end{proof}

\begin{remark}
If $X_1$ is divisible then in the proof of Proposition~\ref{prop add image} we always have $y$ such that $my = \sum_{i=1}^nm_iy_i$. If $T$ is additive and $A$ is convex, we must have $y\in TA$. Hence, in this case we need not assume that $T$ is a bijection. \qede
\end{remark}

For the inverse image, we have a more general result.

\begin{prop}[Convexity under inverse additive map]
Assume that $X_1$ and $X_2$ are monoids and $T:X_1\to X_2$ is additive. Assume that $A\subseteq X_2$ is convex. Then $T^{-1}A\subseteq X_1$ is convex.
\end{prop}

\begin{proof}
Assume that $x_1,\dots,x_n \in T^{-1}A =\{x ~| ~Tx \in A\}$, $m,m_1,\dots,m_n\in \N$ and $x\in X_1$ are such that $mx = \sum_{i=1}^nm_ix_i$, $m=\sum_{i=1}^nm_i$. Since $x_1,\dots,x_n\in T^{-1}A$, we have $Tx_1,\dots,Tx_n\in A$. Since $T$ is additive, we have $\sum_{i=1}^nm_iTx_i = T\left(\sum_{i=1}^nm_ix_i\right) = T(mx) = mTx$. Since $A$ is convex, we have $Tx\in A$. Thus, $x\in T^{-1}A$, which completes the proof. 
\end{proof}

As we shall see, studying convexity in such a general setting also brings about a better understanding of this notion in the standard setting of vector spaces.
One complaint about convexities is that there are too many of them and that in different settings one has to adjoin many additional axioms. This is one more motivation for the current study.

\subsection{Classes of functions}\label{sec conv funs}

Here we consider several classes of functions defined on semimodules,  particularly on monoids, classes which are well studied in the vector spaces setting. In order to define convex functions, we need to consider an \emph{ordered} semimodule, i.e., a semimodule with a partial order $\le$. Given a semimodule $X$ over a semiring $R$, we say that a partial order $\le$ is \emph{compatible} with the module operations, if $rx_1\le rx_2$, $x+x_1\le x+x_2$ for all $x\in X$, $r\in R$, whenever $x_1\le x_2$.

\begin{definition}[Convex function]\label{def convex function}
Let $X,Y$ be a semimodules over a semiring $R$. Assume that $Y$ is equipped with a compatible partial order $\le$. A function $f:X\to Y$ is said to be convex if for every $n\in \N$, every $r_1,\dots,r_n\in R\setminus\{0\}$ and every $x_1,\dots,x_n\in X$, 
\begin{align}\label{ineq conv fun}
rf(x) \le \sum_{i=1}^nr_if(x_i),
\end{align}
for every $x$ satisfying,
\[rx=\sum_{i=1}^nr_ix_i, ~~ r = \sum_{i=1}^nr_i.\]
$f:X\to Y$ is said to be concave if $-f$ is convex. Clearly the sum of two convex functions is convex.
\end{definition}

\begin{remark}
As in Definition~\ref{def convex set}, if we have modules over a ring rather that over a semiring, we assume  we have a partial order on the ring that is compatible with the ring operations, and then in Definition~\ref{def convex function}, we consider only strictly positive elements from the ring. \qede
\end{remark}

\begin{remark}
We often consider a maximal element in $Y$, $\infty$. Also, in the case where $Y$ is a module, not just a semimodule, we may also consider a minimal element $-\infty$. In order for~\eqref{ineq conv fun} to make sense, we  assume for a convex function that $\infty-\infty = 0\cdot \infty = \infty$. \qede
\end{remark}

\begin{definition}[Affine function]
Let $X,Y$ be semimodules over a semiring $R$. 
Then $f:X\to Y$ is said to be affine if for every $n\in \N$, every $r_1,\dots,r_n\in R\setminus \{0\}$ and every $x_1,\dots,x_n\in X$, 
\[rf(x) = \sum_{i=1}^nr_if(x_i),\]
whenever $x\in X$ satisfies,
\[rx=\sum_{i=1}^nr_ix_i, ~~ r = \sum_{i=1}^nr_i.\]
\end{definition}
Clearly every affine function is both convex and concave. For an affine function, we again cannot allow it to attain  $\pm \infty$. 

We can, however, consider the following notion.

\begin{definition}[Generalised affine function]
Assume that $X,Y$ are semimodules over a semiring $R$. Possibly $Y$ contains a maximal element $\infty$ or a minimal element $-\infty$. A function $f:X\to Y \cup\{\pm \infty\}$ is said to be generalized affine if it is both convex and concave. 
\end{definition}

Generalised affine functions are either affine or `very' infinite. 
\begin{prop}\label{prop infty}
Assume that $X$ and $Y$ are groups, and $a:X\to Y\cup\{\pm \infty\}$ is generalised affine. Then either $a$ is everywhere finite, or $a=+\infty$, or $a=-\infty$, or $a$ attains both values $+\infty$ and $-\infty$.
\end{prop}

\begin{proof}
Assume that $a$ is not everywhere finite, and that it is not identically $+\infty$ or $-\infty$. Assume for example that there exist $x_1, x_2\in X$ such that $a(x_1)>\alpha$ for all $\alpha \in \R$ and $a(x_2)$ is finite. We have $2x_2 = \big(x_2+(x_1-x_2)\big)+\big(x_2-(x_1-x_2)\big) = x_1+\big(2x_2-x_1)$, and so since $a$ is concave we have $2a(x_2) \ge a(x_1)+a(2x_2-x_1) > \alpha + a(2x_2-x_1)$. Therefore we must have $a(2x_2-x_1) = -\infty$. If we assume $a(x_1)=-\infty$ rather than $+\infty$, the proof is similar.
\end{proof}

\begin{definition}[Subadditive function]\label{def subadd}
Assume that $X,Y$ are semimodules over a semiring $R$, and assume that $Y$ is equipped with a partial order $\le$. A function $f:X\to Y \cup \{\pm \infty\}$ is said to be subadditive if for every $x,y \in X$, 
\[f(x+y) \le f(x)+f(y).\]
\end{definition}

The function $x\mapsto \sqrt{x}$ is subadditive on $[0,+\infty)$ but not convex.
As we will mostly be concerned with groups and monoids, we  now focus on  functions with subadditive properties over $\N$.

\begin{definition}[$\N$-sublinear functions]\label{def sublin}
Assume that $X,Y$ are semimodules over a semiring $R$, and assume that $Y$ is equipped with a partial order $\le$. A function $f:X\to Y\cup\{\pm \infty\}$ is said to be $\N$-sublinear if it is subadditive and in addition it is positively homogeneous, i.e., $f(mx)=mf(x)$ for every $x\in X$ and every $m\in \N\cup\{0\}$.
\end{definition}

\begin{definition}[Generalised $\N$-linear function]\label{def general lin}
Assume that $X,Y$ are as in Definition~\ref{def sublin}. A function $f:X\to Y\cup\{\pm \infty\}$ is said to be generalised $\N$-linear if both $f$ and $-f$ are $\N$-sublinear.
\end{definition}

If $f$ is a generalised $\N$-linear function and $f$ is finite, then for every choice of positive integers $m_1,\dots,m_n\in \N$, we have $f\left(\sum_{i=1}^nm_ix_i\right) = \sum_{i=1}^nm_if(x_i)$. The functions that satisfy this property are exactly the additive functions on semimodules over $\Z_+$.

If $f$ is $\N$-sublinear and $mx = \sum_{i=1}^nm_ix_i$, where $m=\sum_{i=1}^nm_i$, then 
\begin{align*}
mf(x) = f(mx) = f\left(\sum_{i=1}^nm_ix_i\right) \le \sum_{i=1}^nf(m_ix_i) =\sum_{i=1}^nm_if(x_i).
\end{align*}
In particular, every $\N$-sublinear function on a monoid is convex. Also we have the following.

\begin{prop}\label{prop still conv}
Assume that $X$ is a monoid, $(Y,\le)$ a monoid with a compatible lattice order $\le$, and $f_1,\dots,f_k:X\to Y\cup\{\pm \infty\}$ are convex ($\N$-sublinear, subadditive). Then the function $\max\{f_1,\dots,f_k\}$ is also convex ($\N$-sublinear, subadditive).
\end{prop}

\begin{proof}
If $f_1,\dots,f_k$ are convex and $m,m_1,\dots, m_n\in \N$, $x,x_1,\dots,x_n\in X$ are such that $mx = \sum_{i=1}^nm_ix_i$, $m=\sum_{i=1}^nm_i$, then
\begin{eqnarray*}
m\cdot \max_{1\le j \le k}\{f_j(x)\} & = & \max_{1 \le j \le k}\big\{mf_j(x)\big\}
\\ & \le & \max_{1 \le j \le k}\left\{\sum_{i=1}^nm_if_j(x_i)\right\}
\\ & \stackrel{(*)}{\le} & \sum_{i=1}^nm_i\cdot \max_{1 \le j \le k}\{f_j(x_i)\}.
\end{eqnarray*}
In ($*$) we used the fact that $\le$ is a lattice order, compatible with the group operations on $Y$. The case of sublinear or subadditive functions is easy. We omit the proof. 
\end{proof}

\begin{prop}\label{prop p sufficient}
Assume that $X,Y$ are monoids. Then it suffices in Definition~\ref{def convex function} that $r=p^l$ for a fixed prime $p$ and all $l\in \N$.
\end{prop}
\begin{proof}
Indeed, if $r\neq p^l$, then there exists $l\in \N$ such that $r<p^l$. Thus,
\begin{align*}
\left(p^l-r\right)x + \sum_{i=1}^nr_ix_i =p^lx. 
\end{align*}
By the convexity property,
\begin{align*}
p^lf(x) \le \left(p^l-r\right)f(x) + \sum_{i=1}^nr_if(x_i),
\end{align*}
which gives 
\[rf(x) \le \sum_{i=1}^nr_if(x_i),\] 
as required.
\end{proof}

Proposition~\ref{prop p sufficient} implies the following.
\begin{prop}\label{prop p sufficient sublin}
Assume that $X,Y$ are monoids. Assume that $f:X\to Y$ is subadditive and there exists $p\in\N$ such that $f(px)=pf(x)$ for every $x\in X$, then $f$ is convex. If $Y$ is a group, then $f$ is in fact $\N$-sublinear.
\end{prop}

\begin{proof}
By Proposition~\ref{prop p sufficient}, it is enough to assume in Definition~\ref{def convex function} that $r=p^l$, $l\in \N$. Assume then that $p^lx = \sum_{i=1}^nm_ix_i$. We have,
\begin{align*}
p^lf(x) \stackrel{(*)}{=} f(p^lx) = f\left(\sum_{i=1}^nm_ix_i\right) \stackrel{(**)}{\le} \sum_{i=1}^nm_if(x_i),
\end{align*}
where in ($*$) we used the homogeneity assumption on $f$, and in ($**$) we used the subadditivity of $f$. To prove the second assertion, let $m\in \N$. Then there exist $m',l\in \N$ such that $m+m' = p^l$. Thus, we have
\begin{align*}
(m+m')f(x) = f\big((m+m')x\big) \le f(mx)+m'f(x) \le (m+m')f(x).
\end{align*}
Thus, we have 
\begin{align*}
(m+m')f(x) = f(mx)+m'f(x),
\end{align*}
and since $Y$ is a group, this implies that $f(mx) = mf(x)$ for all $m\in \N$ and all $x\in X$. This complete the proof.
\end{proof}

\subsection{Properties of convex functions}

It is well known  that a convex function on a (semi)normed vector space is continuous at $x_0$ if and only if $f$ is bounded from above in a neighbourhood of $x_0$. If the space is normed, we derive a Lipschitz condition. See~\cite{BV10, Zal02}. We establish a similar fact for convex functions on topological monoids into $[-\infty,\infty]$. For a set $B\subseteq X$ in an additive group and $m\in \N$ define $\frac 1 m B = \big\{x ~ \big| ~ mx \in B\big\}$. It is straightforward to show that if $B$ is convex, $\frac 1 m B$ is convex for all $m\in \N$. Also, a set $B\subseteq X$ is said to be \emph{symmetric} if $-B=B$. Again, if $B$ is symmetric, then $\frac 1 mB$ is symmetric. We have the following.

\begin{prop}\label{lem cont}
Let $X$ be an additive group, $f:X\to [-\infty,\infty]$ a convex function, and assume that there is a symmetric $B\subseteq X$ and $M\in \R$ such that $f(x) \le f(x_0)+M$ for all $x\in x_0+B$. Then for every $y\in \frac 1 m B$, we have $|f(x_0+y)-f(x_0)| \le \frac M m$.
\end{prop}

\begin{proof}
First, note that if $u\in B$ then $-u\in B$ and by convexity we have $2f(x_0) \le f(x_0+u) + f(x_0-u) \le f(x_0+u) +f(x_0)+M$ and so $f(x_0+u) \ge f(x_0)-M$. If $f(x_0)=-\infty$ then $f=-\infty$ on $x_0+B$. Assume then that $f(x_0)>-\infty$. Let $y\in \frac 1 m B$. Then there exists $u\in B$ such that $my =u$. Thus, we have $m(x_0+y) = (x_0+u)+(m-1)x_0$ and then using convexity of $f$ gives $mf(x_0+y) \le f(x_0+u)+(m-1)f(x_0) \le M$. This gives $f(x_0+y)-f(x_0) \le \frac 1 m\big(f(x_0+u)-f(x_0)\big) \le \frac M m$. Also, by convexity, we have $f(x_0) - f(x_0+y) \le f(x_0-y)-f(x_0) \le \frac{M}{m}$, which completes the proof.
\end{proof}
In a topological group the group operations are continuous, and we obtain:

\begin{cor}[Continuity]
Assume that $X$ is a topological group and $f:X\to [-\infty,\infty]$ is convex. Then $f$ is bounded from above in around $x_0$ if and only if $f$ is continuous at $x_0$.
\end{cor}

We next show convex  minorants inherit  continuity of a majorant.

\begin{cor}[Minorants]
Assume that $X$ is a topological group and $f,g:X\to [-\infty,\infty]$. Suppose that $g$ is bounded  above in a neighbourhood of $x_0$, $f$ is a convex minorant of $g$ and $f(x_0)$ is finite. Then $f$ is continuous at $x_0$.
\end{cor}

\begin{prop}[Three-slope lemma for monoids]\label{lem three slope}
Let $X$ be a monoid, and $x,x_1,x_2\in X$, $m_1,m_2 \in \N$ such that $(m_1+m_2)x = m_1x_1+m_2x_2$. Then for any convex function $f:X\to (-\infty,\infty]$ we have

\begin{align*}
\frac{f(x)-f(x_1)}{m_2} \le \frac{f(x_2)-f(x_1)}{m_1+m_2} \le \frac{f(x_2)-f(x_1)}{m_1}.
\end{align*}
\end{prop}

\begin{proof}
By convexity, we have $(m_1+m_2)f(x) \le m_1f(x_1) + m_2f(x_2)$, from which both inequalities follow easily.
\end{proof}

Except in a divisible setting we do not capture convexity using only three points -- we can not induct.

\begin{prop}[Monotone composition]
Assume that $X$ is a monoid. If $f:X\to (-\infty,\infty]$ is sublinear and increasing and $g:X\to (-\infty,+\infty)$ is convex and non-decreasing, then $f\circ g$ is also convex.
\end{prop}

\begin{proof}
Assume that $mx = \sum_{i=1}^nm_ix_i$, $m_i\in \N$, $m=\sum_{i=1}^nm_i$. Then,
\begin{align*}
mf(g(x)) & = f\big(mg(x)\big) \le f\big(m_1g(x_1)+\dots + m_ng(x_n)\big) \le \sum_{i=1}^nm_if(g(x_i)),
\end{align*}
as required.
\end{proof}

\begin{remark}[Midpoint convexity and measurability] 
It is well known that measurability forces a \emph{midpoint convex} function on $\R$ to be convex and an additive function to be linear. There are certainly analogous results to be discovered in appropriate monoids, see for example~\cite{Ros09}.
\qede
\end{remark}

\subsection{Operations on functions}

We next extend some well-known  vector operations on convex and subadditive functions.

\begin{definition}[Subadditive and sublinear minorants]
Assume that $X$ is a monoid and $f:X\to (-\infty,\infty]$. Define 
\begin{align*}
p(x) = \inf\left\{\left.\sum_{i=1}^nf(x_i) ~ \right| ~ \sum_{i=1}^nx_i = x, ~ n\in \N\right\}.
\end{align*}
Then $p$ is the largest function satisfying $p\le f$ and also $p(x+y) \le p(x)+p(y)$. Define also
\begin{align*}
po(x) = \inf\left\{\left.\frac{p(mx)}{m} ~ \right| ~ m\in \N\right\},
\end{align*}
where $p$ is defined as above.
\end{definition}
Now
$po$ is positively homogeneous as we have
\begin{align*}
po(m_0x) & = m_0\inf\left\{\left.\frac{1}{m_0 m}\sum_{i=1}^nf(x_i)~\right|~ m\in \N, ~~ \sum_{i=1}^nx_i = m_0 x\right\}
\\ & = m_0\inf\left\{\left.\frac{1}{m}\sum_{i=1}^nf(x_i)~\right|~ m\in \N, ~~ \sum_{i=1}^nx_i = x\right\},
\end{align*}
where the last equality holds since for every $x_1,\dots, x_n\in X$, we can choose $x'_1,\dots,x'_{n'}\in X$ satisfying $\frac 1 m\sum_{i=1}^nf(x_i) = \frac 1 {m_0m}\sum_{i=1}^{n'}f(x'_i)$. Also $po$ is  subadditive since 
\begin{align*}
\frac 1 {m_1}\sum_{i=1}^nf(x_i)+\frac 1 {m_2}\sum_{i=1}^{n'}f(x'_i)
= \frac 1 {m_1m_2}\left(\sum_{i=1}^nm_2f(x_i)+\sum_{i=1}^{n'}m_1f(x'_i)\right),
\end{align*}
where $\sum_{i=1}^nx_i=m_1x$, $\sum_{i=1}^{n'}x_i'=m_2y$. Choosing a finite index set $I$ which is $m_2$ copies of each $x_i$ for $1\le i \le n$ and $m_1$ copies of each $x'_i$ for $1\le i \le n'$ we get $\sum_{i\in I}x_i = m_1m_2(x+y)$. Thus,
\begin{align*}
\frac 1 {m_1}\sum_{i=1}^nf(x_i)+\frac 1 {m_2}\sum_{i=1}^{n'}f(x'_i) = \frac{1}{m_1m_2}\sum_{i\in I}f(x_i) & \ge po(x+y).
\end{align*}
Taking infima over $m_1$, $m_2$ implies that $po$ is sublinear.

\begin{definition}[$\N$-Sublinear minorant]
Assume that $X$ is a monoid and $f,g:X\to (-\infty,\infty]$. Define 
\begin{align*}
f\wedge g(x) = \inf\left\{\left.\frac{n_1f(x_1)+n_2g(x_1)}{n} ~ \right| ~ n_1x_1+n_2x_2 = nx\right\}.
\end{align*}
\end{definition}
It is straightforward to check that if $f$, $g$ are $\N$-sublinear, so is $f\wedge g$.

\section{Examples}\label{sec examples}

\begin{example}[Vector spaces]
If $X$ is a real vector space, then by definition, $x\in \conv(A)$ if for every $n\in \N$, every $\alpha_1,\dots,\alpha_n\in (0,1)$ and every $x_1,\dots,x_n\in A$,
\begin{align*}
\left(\sum_{i=1}^n\alpha_i\right)x = \sum_{i=1}^n\alpha_ix_i. 
\end{align*}
Taking $\beta_i=\frac{\alpha_i}{\sum\alpha_i}>0$, this is equivalent to 
\[x=\sum_{i=1}^n\beta_ix_i, ~~ \sum_{i=1}^n\beta_i=1,\]
which is the standard definition of a convex hull in a vector space over $\R$.
 \qede
\end{example}

\begin{example}[$\R$ as a $\mathbb Q$-module]\label{ex r q}
Consider $X=\R$ as a vector space over $\mathbb Q$. In such case $x\in \conv(A)$ if for every $n\in \N$, every $q_1,\dots,q_n\in \mathbb Q_+\setminus\{0\}$ and every $x_1,\dots,x_n\in A$, 
\begin{align*}
qx=\sum_{i=1}^nq_ix_i, ~~ q=\sum_{i=1}^nq_i,
\end{align*}
which is equivalent to 
\begin{align*}
x = \sum_{i=1}^nq'_ix_i, ~~ \sum_{i=1}^nq'_i=1, ~~ q'_i\in [0,1]\cap \mathbb Q,
\end{align*}
i.e., we take only \emph{rational} convex combinations.
\qede
\end{example}

We now present  examples of monoids and of the behaviour of the hull operator.
 
\begin{example}[The lattice $\Z^d$]\label{ex z}
Consider $X=\Z^d$ with the addition induced from $\R^d$. For every $A\subseteq  X$, we have
\begin{align}\label{conv in z}
\conv_{\Z^d}(A) = \conv_{\R^d}(A)\cap \Z^d,
\end{align}
where $\conv_{\R^d}(A)$ is the standard convex hull of $A$ in $\R^n$. To see this, first note that if $x\in \conv_{\Z^d}(A)$, then there exist $x_1,\dots,x_n\in A$, and $m_1,\dots,m_n,m\in \N$ such that $mx = \sum_{i=1}^n m_ix_i$, $m=\sum_{i=1}^nm_i$. This implies that 
\[x=\sum_{i=1}^n\frac{m_i}{m}x_i, ~~\sum_{i=1}^n\frac{m_i}{m}=1,\] 
which means that $x\in \conv_{\R^d}(A)$, and so $\conv_{\Z^d}(A)\subseteq \conv_{\R^d}(A)\cap \Z^d$. To prove to other inclusion, use induction on the dimension. If $d=1$, and $x\in \conv_{\R}(A)\cap \Z$, then $x$ is an integer which is also a convex combination of two other integers $x_1,x_2$. Therefore, we can write $x=q_1x_1+q_2x_2$ with $q_1,q_2\in \mathbb Q$, and so there exist $m_1,m_2,m\in \Z$ such that $mx = m_1x_1+m_2x_2$ and $m=m_1+m_2$. To prove the general case, assume that $x\in \conv_{\R^d}(A)\cap \Z^d$. Then there exist $x_1,\dots,x_n\in A$ and $\alpha_1,\dots,\alpha_n\ge 0$ with $\sum_{i=1}^n \alpha_i=1$ such that $x=\sum_{i=1}^n \alpha_ix_i$. By Carath\'eodory's Theorem \cite{Mat02}, we can write $x = \sum_{i=1}^{n'}\alpha_j x_j$, where $n'\le d+1$ (we might have to rearrange the points $x_1,\dots,x_n$). If $\mathrm{dim}\big(\mathrm{span}\{x_1,\dots,x_{n'}\}\big)<d$, use the 
induction hypothesis to conclude that we can write $x=\sum_{i=1}^{n'}q_ix_i$, with $q_i\in \mathbb Q$. Otherwise, we have the following linear system.
\begin{align*}
\left[ \begin{array}{cccc}
x_1 & x_2 & \dots & x_{d+1} \\
  1  & 1   & \dots & 1
\end{array}
\right]
\left[ \begin{array}{c}
\alpha_1 \\ \vdots \\ \alpha_{d+1}
\end{array}
\right] = x.
\end{align*}
where $x_1,\dots,x_{d+1}$ are written as column vectors. In this case, one can show that the system has a \emph{unique} solution. Thus the matrix is invertible. Since the matrix has integer coefficients, it follows that the $q_i$'s are rational. And so once again we can write $x=\sum_{1}^{n'}q_ix_i$ with $q_i\in \mathbb Q$, which implies that $x\in \conv_{\Z^d}(A)$. \qede
\end{example}

\begin{example}[General lattices in $\R^d$] 
We say that $v_1,\dots,v_k\in \R^d$ are independent over $\Z$ if
\begin{align*}
\sum_{i=1}^km_iv_i = 0, ~~m_i\in \Z \Longrightarrow m_i = 0.
\end{align*}
Assume that $\Gamma = \spn_{\Z}\{v_1,\dots,v_k\}$, where $v_1,\dots,v_k\in \R^d$ are independent over $\Z$. Let $T:\R^k\to \R^d$ be defined as 
\[T(\alpha_1,\dots,\alpha_k) = \sum_{i=1}^k \alpha_iv_i.\]
$T$ is linear and $T(\Z^k) = \Gamma$. Also, since $v_1,\dots,v_k$ are independent over $\Z$, it follows that $T\big|_{\Z^k}$ is invertible. Finally, since $\Gamma$ is a $\Z$-module, it follows from Proposition~\ref{hull monoid} that
\begin{eqnarray*}
\conv_{\Gamma}(A) =  \left\{\left.\sum_{i=1}^nq_ia_i ~\right|~ a_i\in A, q_i\in \Q\cap[0,1], \sum_{i=1}^nq_i=1 \right\}.
\end{eqnarray*}
Hence,
\begin{eqnarray*}
\conv_{\Gamma}(A)& = & T\left(\conv_{\Z^k}\left(T^{-1}A\right)\right)
\\ & \stackrel{(*)}{=} & T\left(\conv_{\R^k}\left(T^{-1}A\right)\cap \Z^k\right)
\\ & \stackrel{(**)}{=} & T\left(\conv_{\R^k}\left(T^{-1}A\right)\right)\cap T\Z^k
\\ & \stackrel{(***)}{=} & \conv_{\R^d}(A)\cap \Gamma,
\end{eqnarray*}
where in ($*$) we used Example~\ref{ex z}, in ($**$) we used the invertibility of $T$ over $\Z^k$, and in ($***$) we used the linearity of $T$. \qede
\end{example}

\begin{example}[Dyadic rationals]
Let $X$ be the rational numbers of the form $\frac m {2^n}$, where $m,n\in \Z$. We have that $X$ is 2-semidivisible as $X = 2X$, since $\frac m {2^n} = 2\frac m {2^{n-1}}$, but for any odd number $k$ we do not have $1 = k\cdot \frac {m}{2^n}$. Thus, $X$ is not divisible. \qede
\end{example}
 
\begin{example}[Arctan semigroup] Let $X=([0,\infty),\oplus)$ with addition defined by \[a\oplus b = \frac{a+b}{1+ab}.\]
Note that if $a,b\neq 0$ then $a\oplus b = \frac 1 a \oplus\frac 1 b$. The unit is $0$ as $a\oplus 0 = a$. Also, for all $a\ge 0$, $a\oplus 1 = 1$. Hence, $\conv(\{0\})=\{0\}$ and $\conv(\{1\}) = \{1\}$. For every $a>0$ we have $a\oplus a = \frac 1 a \oplus \frac 1 a$. Thus, if $a\neq 1$ then $\frac 1 a \in \conv(\{a\})$. This means that $\{0\}$ and $\{1\}$ are the only convex singletons. Also, since $a\oplus 1=1$ for every $a\in X$, then for every $A\subseteq X$, we have
\[\conv(A\cup\{1\}) = \conv(A)\cup\{1\}.\]
Finally, note that for every $a\ge 0$, we have
\begin{align*}
3a = a\oplus a\oplus a = \frac{3a+a^3}{1+3a^2},
\end{align*}
and the function $a\mapsto \frac{3a+a^3}{1+3a^2}$ is onto $[0,\infty)$. Thus, $X$ is 3-semidivisible. On the other hand, $a\oplus a = \frac{2a}{1+a^2} \le 1$, and so $X$ is not divisible. In fact is is divisible precisely for all odd numbers. \qede
\end{example}

The  next example illustrates that  finding convex or affine functions on a group is solving potentially subtle functional equations and inequalities

\begin{example}[Hyperbolic group]
Let $X_p$ be the collection of all $2\times 2$ symmetric matrices of the form $e^{\frac{2\pi i l}{p}} M(\theta)$, where $M(\theta) = \left[\begin{array}{c c} \cosh(\theta) & \sinh(\theta) \\ \sinh(\theta) & \cosh(\theta)\end{array}\right]$, $\theta\in \R$ and $0\le l \le p-1$. Then $X_p$ is a group under the standard matrix multiplication, as we have 
\[\left(e^{\frac{2\pi i l_1}p}M(\theta_1)\right)\cdot \left(e^{\frac{2\pi il_2}p}M(\theta_2)\right) = e^{\frac{2\pi i (l_1+l_2)}p}M(\theta_1+\theta_2).\] 
In particular, the group is commutative. Also, if $p|n$, we have that $M(\theta)^{n} = \left(e^{\frac{2\pi il}{p}}M(\theta)\right)^{n} = M(n\theta)$ for all $0\le l \le p-1$. Thus, in this case we have $nX_p\subsetneq X_p$. Otherwise, if $p\nmid n$, then we have $\left(e^{\frac{2\pi l}p}M(\theta)\right)^{n} = e^{\frac{2\pi nl}p}M(n\theta)$. Since $\theta\mapsto n\theta$ and $e^{\frac{2\pi l}{p}}\mapsto e^{\frac{2\pi n l}{p}}$ is one-to-one and onto (the second since $p\nmid n$), it follows that in this case $nX_p=X_p$. Altogether, we conclude that $X_p$ is $n$-divisible if and only if $p\nmid n$. 

Next, we would like to show that it is easy to produce convex functions on the group $X_p$. Indeed, if $f:\R\to \R$ is a convex function then defining $F\big(e^{\frac{2\pi i l}{p}}M(\theta)\big) = f(\theta)$ is also convex. To see this, for $m_1,\dots,m_n\in \N$ and $x_1,\dots,x_n,x\in X$ satisfying $mx = \sum_{i=1}^nm_ix_i$, $m=\sum_{i=1}^nm_i$, assume that $x = e^{\frac{2\pi i l}{p}}M(\theta), x_1 = e^{\frac{2\pi i l_1}{p}}M(\theta_1), \dots, x_n = e^{\frac{2\pi i l_n}{p}}M(\theta_n)$. Thus, we have 
\begin{align}\label{equality matrices}
e^{\frac{2\pi i ml}{p}}M(m\theta) = e^{\frac{2\pi i }{p}\sum_{j=1}^nm_jl_j}M(m_1\theta_1+\dots m_n\theta_n).
\end{align} 
Note that if $e^{\frac{2\pi i l}{p}}M(\theta)$ is the identity matrix, then $l=\theta=0$. Therefore, if $e^{\frac{2\pi i l_1}{p}}M(\theta_1) = e^{\frac{2\pi i l_1}{p}}M(\theta_1)$, then $l_1=l_2$ and $\theta_1=\theta_2$. In particular, ~\eqref{equality matrices} implies that $m\theta = \sum_{i=1}^nm_i\theta_i$. Hence, we have
\begin{align*}
mF(x) & = mf(\theta) \le \sum_{i=1}^n m_if(\theta_i)= \sum_{i=1}^nm_iF(x_i).
\end{align*}

Note that restriction to $M(\theta)$ (determinant one) is a divisible subgroup. Also, consider the group 
\begin{align*}
X_{\R} = \Big\{ e^{it}M(\theta)~\Big|~ t,\theta \in \R\Big\},
\end{align*}
again with the standard multiplication. Then $X_{\R}$ is a divisible group, since for every $t,\theta\in \R$ and every $n\in \N$, we have
\begin{align*}
e^{it}M(\theta) = \left(e^{i\frac t n}M(\theta/n)\right)^n.
\end{align*}
Note that for every $p$, $X_p$ is a semidivisible subgroup of $X_{\R}$. Finally, note that if we consider $X_{\R}$ as a topological space, equipped with the topology induced from $\R^4$, then $X_{\R}$ is connected since we can write $X_{\R} = \Phi(\R^2)$, where $\Phi:(t,\theta)\mapsto e^{it}M(\theta)$ is continuous. See~\cite{BG15} for a more detailed discussion on convexity in topological groups. \qede
\end{example}

\begin{example}[Finite groups]\label{ex finite}
If $X$ is a finite group then by the pigeon hole principle there exists $m\in \N$ such that $mx=0 = m\cdot 0$. Thus $x\in \mathrm{conv}(\{0\})$ for every $x\in X$. Hence, $X$ and $\emptyset$ are the only convex sets in $X$. \qede
\end{example}

\begin{example}[Circle group] Let $\T=\R/\Z$ with the standard coset addition. In this case, if $x= [m/n]$ $m,n\in \N$ then $nx = [0]$. 
Thus, 
\[\conv(\{0\}) = \big\{x\in \T ~\big| ~ x \text{ has finite order}\big\}.\]
Also, for every $x\in X$, $x+y\in \conv(\{x\})$ for every $y\in X$ which is of finite order. Thus, there are no convex singletons in $X$. \qede
\end{example}

\begin{example}[Pr\"ufer group]
This is a subgroup of the circle group $\T$, which is given by
\begin{align*}
\Z(p^{\infty}) = \big\{\exp\big(2\pi im / p^n\big)~\big|~m,n\in \N\cup\{0\}\big\},
\end{align*}
i.e., all $p^n$-th roots of unity. Every element in this group has a finite order and so by the previous example (and also by example~\ref{ex finite}), the only two convex sets are $\emptyset$ and the entire group. It is also known that $\Z(p^{\infty})$ is divisible. To see this, note that it is enough to show that $X=qX$ for every prime $q$. Let $x=\exp\big(2\pi i m/p^n\big)$. If $n=0$ then $x=1=1^q$. Assume then that $n>0$. If $q=p$ then $x = y^q$ where $y=\exp\big(2\pi im/p^{n+1}\big)$. If $q\neq p$ then since the greatest common divisor of $p^n$ and $q$ is 1, there exist $a,b\in \Z$ such that $ap^n+bq=1$. So $x = x^{ap^n+bq} = x^{ap^n}x^{bq} = x^{bq}$. Choosing $y=x^b$, then $x=y^q$, as needed.
\qede
\end{example}

\begin{example}[Extensions of $\Q$]
Consider $X = \Q+\theta \Q$, where $\theta$ is irrational, with the addition operation then the mapping $\Phi:a+\theta b \mapsto (a,b)$ is a group homomorphism from $X$ to $\Q^2$. Thus 
\[\conv_{X}(A) = \Phi^{-1}\big(\conv_{\Q^2}(\Phi(A))\big).\]
Similarly, we can consider extensions of $\Q$ be any number of algebraically independent numbers. \qede
\end{example}

\begin{example}[Half line with multiplication]
If $X=((0,\infty),\cdot)$, this semigroup is isomorphic to $(\R,+)$ via $x\mapsto \log(x)$. Thus,
\begin{align}\label{conv log}
\conv_X(A) = \exp\big(\conv_{(\R,+)}(\log(A))\big).
\end{align}
If instead we choose $X=([0,\infty),\cdot)$, then if $0\in A$, we have
\[\conv_X(A) = \{0\}\cup \exp\big(\conv_{(\R,+)}(\log(A))\big),\] 
if $0\notin A$ then~\eqref{conv log} still holds. 
\qede
\end{example}

\begin{example}[$\sigma$-algebras with symmetric differences]
Given a set $S$, let $X$ be a $\sigma$-algebra of subsets of $S$. For $A,B\in X$, let $A+B = A\triangle B = (A\cup B)\setminus (A\cap B)$. Clearly $A\triangle B = B\triangle A$. Also, note that for every $A\in \mathcal F$, $A\triangle \emptyset =A$, and $A\triangle A = \emptyset$. Thus, $\emptyset$ is the additive unit and $A = -A$. It also follows that from every $A\in X$ and $n\in \mathbb N$, $2nA=\emptyset$ and $(2n-1)A = (2nA)\triangle A = \emptyset\triangle A = A$. Thus, $2nX= \{\emptyset\} \subsetneq X$ and $(2n-1)X=X$, and so $X$ is ($2n-1$)-semidivisible but not $2n$-semidivisible. Next, assume that $A_1,\dots,A_n, A\in X$ and $m_1,\dots,m_n,m\in \N$ are such that $mA = \sum_{i=1}^nm_iA_i$ and $m = \sum_{i=1}^nm_i$. Then by the above arguments we have in fact
\begin{align*}
mA = \sum_{i: 2\nmid m_i}m_iA_i = \sum_{i: 2\nmid m_i}A_i.
\end{align*}
Thus, if $\mathcal A\subseteq X$, then we can write
\begin{align*}
\conv(\mathcal A) = \left\{A\subseteq X~\left| ~A = \sum_{i=1}^nA_i, ~~ A_i\in \mathcal A, ~~ n\in \N\right.\right\}.
\end{align*}
Note that we always have $\emptyset\in \conv(\mathcal A)$ since $A+ A = \emptyset = 2\emptyset$. This group can also be studied as a topological group. See~\cite{BG15}.
\qede
\end{example}
 
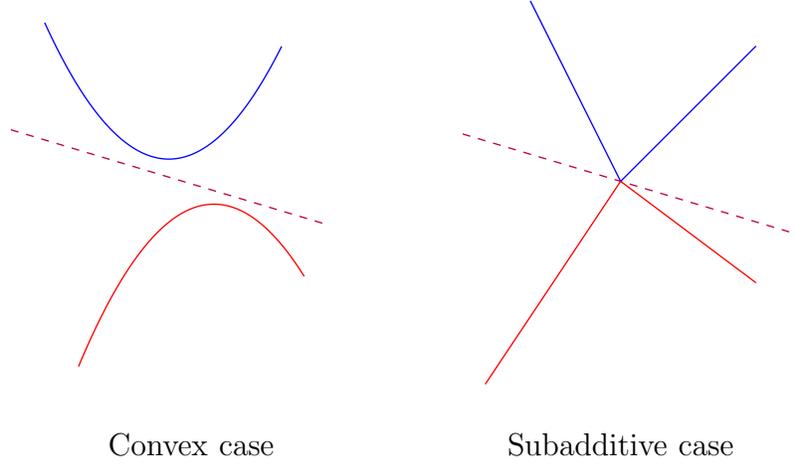
\begin{figure}[t]
\begin{center}
\begin{tikzpicture}
  \draw[line width=0.5pt, scale=0.3, draw=blue] plot[smooth, samples=100, domain=-6.5:4] ({\x},{0.2*(\x+1)*(\x+1)+1});
  \draw[line width=0.5pt, scale=0.3, draw=red]  plot[smooth, samples=100, domain=-5:5] ({\x},{-0.2*(\x-1)*(\x-1)-1});
  \draw[line width=0.5pt, dashed, scale=0.3, draw=purple]  plot[smooth, samples=100, domain=-8:6] ({\x},{-0.3*\x-0.1});
  \node [below=6pt] at (0,-3) {Convex case};
\end{tikzpicture}
\qquad \qquad
\begin{tikzpicture}
  \draw[line width=0.5pt, scale=0.3, draw=blue] plot[smooth, samples=100, domain=0:6] ({\x},{\x});
  \draw[line width=0.5pt, scale=0.3, draw=blue] plot[smooth, samples=100, domain=-4:0] ({\x},{-2*\x});
  \draw[line width=0.5pt, scale=0.3,draw=red]  plot[domain=0:6,smooth, samples=100] ({\x},{-0.75*\x});
  \draw[line width=0.5pt, scale=0.3, draw=red]  plot[domain=-6:0,smooth, samples=100] ({\x},{1.5*\x});
  \draw[line width=0.5pt, dashed, scale=0.3,draw=purple]  plot[domain=-7:8,smooth, samples=100] ({\x},{-0.3*\x});
  \node [below=6pt] at (0,-3) {Subadditive case};
\end{tikzpicture}
\end{center}
\caption{Separation in groups}\label{fig:groups}
\end{figure}

\section{Interpolation of scalar-valued functions}\label{sec interpolate}
We begin with a slight extension of a seminal result.

\begin{thm}[Kaufman \cite{Kau66}]\label{thm kaufman}
Let $X$ be a monoid and $f,g:X\to [-\infty,\infty)$ satisfying $g\le f$, where $f$ and $-g$ are subadditive. Then there exists a function $a:X\to \R$ which is additive and satisfies $g\le a\le f$.
\end{thm}
Theorem~\ref{thm kaufman} is a generalization  of Kaufman's Hahn-Banach result which itself extends the seminal  result by Mazur and Orlicz~\cite{MO53}. Under the assumption that $X$ is semidivisible, the following holds.

\begin{thm}[Interpolation of convex functions]\label{thm separate}
Assume that $X$ is a semidivisible monoid, and $f:X\to [-\infty,\infty]$ and $-g: X\to [-\infty,\infty]$ are convex. Then there exists a function $a:X\to [-\infty,\infty]$ which is generalised affine and satisfies $g\le a \le f$.
\end{thm}

We illustrate the two results in Figure \ref{fig:groups}.

\begin{proof}
First, since $f$, $-g$ are convex and $g\le f$, we have 
\begin{equation}\label{ineq conc}
\begin{aligned}
& mf(x) \ge \sum_{i=1}^n m_ig(x_i), &
\\& mx = \sum_{i=1}^n m_ix_i, ~~m =\sum_{i=1}^n m_i. &
\end{aligned}
\end{equation}
If $f=g$, then $f$ is generalised affine and the proof is complete. Assume then that there exists $x_0\in X$ and $r\in \R$ such that $f(x_0)>r>g(x_0)$. In such case, either we have
\begin{align}\label{case 1}
mf\left(x\right) \ge m_0r+\sum_{i=1}^nm_ig(x_i), 
\end{align} 
whenever we have
\begin{align*}
 mx =m_0x_0 + \sum_{i=1}^n m_ix_i, ~~m =m_0+\sum_{i=1}^n m_i,
\end{align*}
or else
\begin{align}\label{case 2} 
(m'-m_0') f(y)+m_0'r \ge \sum_{i=1}^{n'}m'_ig(y_i), 
\end{align} 
whenever we have
\begin{align*}
m_0'x_0+(m'-m_0')y =\sum_{i=1}^{n'} m'_iy_i, ~~  m' = \sum_{i=1}^{n} m'_i, ~~ m_0' \le m'.
\end{align*}
To see this, assume that neither~\eqref{case 1} nor~\eqref{case 2} hold. Multiplying~\eqref{case 1} by $m'$ and~\eqref{case 2} by $m$, we can find integers $m_0,\dots,m_n, m_0',\dots,m_{n'}'\in \N$ and elements $x_1,\dots,x_n,y_1,\dots,y_{n'}\in X$ satisfying
\begin{align}
\label{eq m} & m =m_0 + \sum_{i=1}^n m_i, ~~~ mx =m_0x_0+\sum_{i=1}^n m_ix_i,
\\
\label{eq m'} & m' = \sum_{i=1}^{n'} m'_i, ~~~ m_0' \le m', ~~~ m_0'x_0+(m'-m_0')y =\sum_{i=1}^{n'} m'_iy_i,
\end{align}
such that
\begin{align*}
m_0'\sum_{i=1}^nm_ig(x_i) + m_0\sum_{i=1}^{n'}m'_ig(y_i) & > m_0'm f(x)+ m_0(m'-m_0') f(y) 
\\ & \ge \big(m_0'm+m_0(m'-m_0')\big) f(z), 
\end{align*} 
where $z$ satisfies 
\begin{eqnarray}\label{prop z} 
\nonumber \big(m_0'm+m_0(m'-m_0')\big) z & = & m_0'mx+m_0(m'-m_0')y
\\\nonumber & \stackrel{\eqref{eq m}}{=} & m_0'm_0x_0+m_0'\sum_{i=1}^nm_ix_i+m_0(m'-m_0')y
\\ & \stackrel{\eqref{eq m'}}{=} & m_0\sum_{i=1}^{n'}m'_iy_i+m'_0\sum_{i=1}^nm_ix_i.
\end{eqnarray}
Such $z$ always exists since $X$ is semidivisible, i.e., $X = p^lX$ for some prime $p$ and $l\in \N$ and by Proposition~\ref{prop p sufficient} we may assume that $m_0'm+m_0(m'-m_0') = p^l$. Now, we have
\begin{eqnarray*}
m_0'\sum_{i=1}^n m_i+m_0\sum_{i=1}^{n'} m'_i & \stackrel{\eqref{eq m}\wedge\eqref{eq m'}}{=} & m'_0(m-m_0)+m_0m' 
\\ & = & m_0'm-m_0'm_0+m_0m' 
\\ & = & m_0'm + m_0(m'-m_0').
\end{eqnarray*}
Hence, we have
\begin{eqnarray}\label{contradict}
\nonumber \big(m_0'm + m_0(m'-m_0')\big)f(z) & \stackrel{(*)}{\ge} &\big(m_0'm + m_0(m'-m_0')\big) g(z) 
\\& \stackrel{(**)}{\ge} & m_0'\sum_{i=1}^n m_ig(x_i) + m_0\sum_{i=1}^{n'} m'_ig(y_i), 
\end{eqnarray} 
where in ($*$) we used the fact that $g\le f$ and in ($**$) we used the fact that $g$ is concave. Now,~\eqref{contradict} is a contradiction to~\eqref{ineq conc}. Thus, we must have that either~\eqref{case 1} or~\eqref{case 2} hold. Assume first that~\eqref{case 1} holds. Define 
\begin{align}\label{def h}
h(x) = \sup\left[\frac 1 {k}\left(k_0r+\sum_{i=1}^nk_ig(x_i)\right)\right],
\end{align} 
where the supremum is taken over all $k,k_0,k_1,\dots,k_n\in \N$ and $y_1,\dots,y_n\in X$ such that $kx = k_0x_0+\sum_{i=1}^nk_iy_i$ and $k= k_0+\sum_{i=1}^nk_i$. By choosing $k_1=\dots=k_n=0$, we have $h(x_0) \ge r > g(x_0)$. Since $g$ is concave we also have that $h\ge g$, and by~\eqref{case 1} it follows that $h \le f$. Next, we would like to show that $h$ is concave, and that~\eqref{ineq conc} holds for $h$ instead of $g$. To show the concavity, let $m_1,\dots,m_n\in \N$, and $x_1,\dots,x_n,x\in X$ such that $mx = \sum_{i=1}^nm_ix_i$ and $m=\sum_{i=1}^nm_i$. Let $\epsilon>0$, and for each $1\le i \le n$, choose $k_i, k_{i,0},\dots,k_{i,n_i}\in \N$ and $y_{i,1},\dots,y_{i,n'}\in X$ such that $k_ix = k_{i,0}x_0+\sum_{j=1}^{n_i}k_{i,j}y_{i,j}$, $k_i=k_{i,0}+\sum_{j=1}^{n_i}k_{i,j}$ such that 
\begin{align}\label{ineq eps}
k_ih(x_i)-\frac{k_i\epsilon}{m} \le k_{i,0}r+\sum_{j=1}^{n_i}k_{i,j}g(y_{i,j}).
\end{align}
Now, we have
\begin{align*}
\left(m\prod_{i=1}^nk_i\right)x & = \sum_{i=1}^n\left(m_i\prod_{j\neq i}k_j\right)k_ix_i 
\\ & =  \sum_{i=1}^n\left(m_i\prod_{j\neq i}k_j\right)\left(k_{i,0}x_0+\sum_{j=1}^{n_i}k_{i,j}y_{i,j}\right)
\\ & = \sum_{i=1}^n\left(m_ik_{i,0}\prod_{j\neq i}k_j\right)x_0+\sum_{i=1}^n\left(m_i\prod_{j\neq i}k_j\right)\left(\sum_{j=1}^{n_i}k_{i,j}y_{i,j}\right).
\end{align*}
Also, we have
\begin{align*}
m\prod_{i=1}^nk_i = \sum_{i=1}^n\left(m_ik_{i,0}\prod_{j\neq i}k_j\right) + \sum_{i=1}^n\left(m_i\prod_{j\neq i}k_j\right)\sum_{j=1}^{n_i}k_{i,j}.
\end{align*}
Thus, by the definition of $h$~\eqref{def h}, we have
\begin{eqnarray*}
&& mh(x) \ge \frac{1}{\prod_{i=1}^nk_i}\left[ \sum_{i=1}^n\left(m_ik_{i,0}\prod_{j\neq i}k_j\right)r + \sum_{i=1}^n\left(m_i\prod_{j\neq i}k_j\right)\sum_{j=1}^{n_i}k_{i,j}g(y_{i,j})\right]
\\ && = \frac{1}{\prod_{i=1}^nk_i}\sum_{i=1}^nm_i\prod_{j\neq i}k_j\left[k_{i,0}r + \sum_{j=1}^{n_i}k_{i,j}g(y_{i,j})\right]
\\ && \stackrel{\eqref{ineq eps}}{\ge}  \frac{1}{\prod_{i=1}^nk_i}\sum_{i=1}^n\left(m_i\prod_{j\neq i}k_j\right)\left(k_ih(x_i)-\frac{k_i\epsilon}{m}\right)
\\ && = \sum_{i=1}^nm_ih(x_i)-\epsilon.
\end{eqnarray*}
Since $\epsilon$ is arbitrary, it follows that $h$ is concave. Finally, we would like to show that if $mx = \sum_{i=1}^nm_ix_i$, $m=\sum_{i=1}^n$, then $\sum_{i=1}^nm_ih(x_i) \le mf(x)$. This follows from the fact that $h$ is concave together with the fact that $h\le f$. The existence and the properties of $h$ show that $g$ is \emph{not} the maximal element in the class of all concave functions that satisfy~\eqref{ineq conc}. Analogously, if~\eqref{case 2} holds, define 
\begin{align}\label{def k}
h'(x) = \inf\left[ \frac 1 {k'} \big(k'_0r +(k'-k_0')f(y)\big)\right],
\end{align}
where the infimum is taken over all $k_0'\ge 0$, $k'\in \N$, and $y\in X$ such that $k'x = k_0'x_0+(k'-k_0')y$. If $k'=k_0'$ we define the right side of~\eqref{def k} to be $r$. Choosing $k'=k_0'$ gives $h'(x_0) \le r < f(x_0)$ and choosing $k_0=0$ gives $h'(x) \le f(x)$ for all $x\in X$. Since~\eqref{case 2} holds and $g$ is concave, we also have that $g\le h'$ and~\eqref{ineq conc} holds with $h'$ instead of $f$. Also, in an analogous way to the previous case, one can show that $h'$ is convex. To conclude the proof, define the following ordered set $\mathcal D$ of all pairs of the form $(h,h')$, where $h$ is concave, $h'$ is convex, and~\eqref{ineq conc} holds if we replace $g$ by $h$ or $f$ by $h'$. Define the partial order on $\mathcal D$ to be $(h,h') \le (w,w') \iff h\le w \text{ and } w'\le h'$. Since $(g,f)\in \mathcal D$, this chain is non-empty and therefore has a maximal element. By the above consideration we conclude  the maximal element is generalised affine.
\end{proof}

\begin{remark}
Note that we used the semidivisibility only to show that either~\eqref{case 1} or~\eqref{case 2} must hold. We did not use this fact again in the proof.
\qede
\end{remark}

\begin{remark}
Similarly, the results  hold if we work in a semimodule. \qede
\end{remark}

\begin{figure}[t]
\begin{center}
\begin{tikzpicture}
  \draw[draw=red!50!white, fill=red!50!white] plot[smooth, samples=100, domain=-2.5:0](\x,{-sqrt(-\x)}) -- plot[smooth, samples=100, domain=0:-2.5] (\x,{2.5});
  \draw plot[smooth, samples=100, domain=-2.5:0](\x,{-sqrt(-\x)});
  \draw[draw=blue!50!white, fill=blue!50!white] plot[smooth, samples=100, domain=0:2.5](\x,{sqrt(\x)}) -- plot[smooth, samples=100, domain=2.5:0] (\x,{-2.5});
  \draw plot[smooth, samples=100, domain=0:2.5](\x,{sqrt(\x)});
  \draw[-] (0,-2.5) -- (0,2.5);
  \node at (0,0) {\textbullet};  
\end{tikzpicture}
\end{center}
\caption{Failure of finite affine separation}\label{fig:root}
\end{figure}
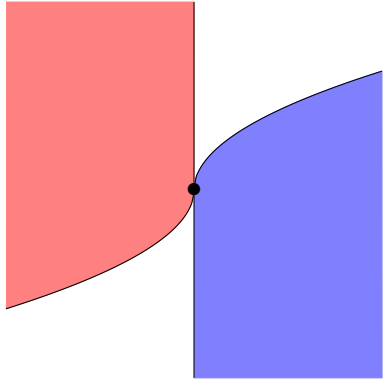

\begin{remark}
In general we cannot expect the affine function $a$ to be better than generalised affine in Theorem~\ref{thm separate}, even if $X$ is a vector space. This is illustrated by the example  of $f(x) = -\sqrt x$ if $x\ge 0$ and $f(x) = -\infty$ if $x<0$ and $g(x) = -f(-x)$, where $X=\R$. The only separator comes from letting $a$ to be $+\infty$ when $x>0$, $-\infty$ when $x<0$ and $0$ when $x=0$. See Figure \ref{fig:root}. \qede
\end{remark}

On the other hand, using Proposition~\ref{prop infty}, we have the following.

\begin{cor}
Assume that $X$ is a group. If either $f$ or $g$ is everywhere finite and the other function is somewhere finite, then $a$ is finite and affine.
\end{cor} 

The vector space version of the following result is used in \cite{Hol75} as the basis for Hahn-Banach theory. Once established, one imposes additional core conditions on $A,B$ to show  ${\rm cl}~C \cap{\rm  cl}~ D$ is a separating half-space. Here one uses the algebraic closure. We take a different (more modern) approach in the next section.

\begin{cor}[Stone's lemma for monoids]
Assume that $X$ is a semidivisible monoid and $A,B\subseteq X$ are disjoint convex sets. Then there exist $C,D\subseteq X$ disjoint and convex such that $A\subseteq C$, $B\subseteq D$ and $C\cup D = X$.
\end{cor}

\begin{proof}
Let $f = \iota_A$, $g=-\iota_B$, where 
\[\iota_A(x) = \begin{cases}0& x\in A \\ \infty & x\notin A\end{cases},\]
and similarly for $\iota_B$. Then $f,-g:X\to [-\infty,\infty]$ are convex. Use Theorem~\ref{thm separate} to deduce the existence of a generalised affine function $a:X\to [-\infty,\infty]$ with $-\iota_B \le a \le \iota_A$. Choosing 
\[C = \big\{x\in X~\big|~ a(x)<0\big\}, ~~D = \big\{x\in X~\big|~ a(x)\ge 0\big\},\] 
concludes the proof.
\end{proof}

Theorem~\ref{thm separate} also implies the following.

\begin{cor}
Assume that $X$ is semidivisible monoid and $f:X\to [-\infty,\infty]$ is convex. Then $f$ is the supremum over its generalised affine minorants.
\end{cor}

\begin{proof}
Clearly we have
\begin{align}\label{ineq always}
f(x) \ge \sup\big\{a(x)~\big|~a \le f,~a\text{ is affine}\big\}.
\end{align}
To show that equality holds, assume to the contrary that we have a strict inequality in~\eqref{ineq always}. The function $g$ which equals the supremum at $x$ and $-\infty$ everywhere else is concave. By Theorem~\ref{thm separate}, there exists an affine function $a$ such that 
\begin{align*}
\sup\big\{a(x)~\big|~a \le f,~a\text{ is affine}\big\} < a(x) < f(x),
\end{align*}
which is a contradiction.
\end{proof}
 
\begin{example}[Non separation]
In the non-divisible setting, Theorem~\ref{thm separate} fails even  for everywhere finite functions. Take for example $X=\Z^2$. Let $A=\conv_{\R^2}\big(\big\{(0,2), (1,0)\big\}\big)$ and $B=\conv_{\R^2}\big(\big\{(0,1), (2,0)\big\}\big)$, and 
\begin{align*}
& f(x) = 2\sqrt 5d_A(x)-1,
\\
& g(x) = -2\sqrt 5d_B(x)+1.
\end{align*}
where $d_A(x) = \inf_{a\in A}\|x-a\|_{\R^2}$. Note that for every $x\in \Z^2$ such that $x\notin A$, we have $d_A(x) \ge \frac 1 {\sqrt 5}$. Similarly, if $x\in \Z^2$ and $x\notin B$, we have $d_B(x) \ge \frac 1 {\sqrt 5}$. For every $x\in \Z^2$, either $x\notin A$ or $x\notin B$ and so $d_A(x)+d_B(x) \ge \frac 1 {\sqrt 5}$. Hence,
\begin{align*}
f(x)-g(x) = 2\sqrt 5 \big(d_A(x)+d_B(x)\big)-2 \ge 0,
\end{align*}
and so $g\le f$ on $\Z^2$. Also, $f$ and $-g$ are convex, since they are convex on all of $\R^2$ (the distance to a convex set in a vector space is a convex function). Assume that $a$ is affine and satisfies $g\le a \le f$. By the choice of $f$ and $g$, $a$ has to be finite everywhere. Since $a$ is affine, we can write $a(m_1,m_2) = c+\alpha_1m_1+\alpha_2m_2$, where $c,\alpha_1, \alpha_2\in \R$. Since $a \le f$, we can choose $x=(0,2)$ and $x=(1,0)$ and obtain
\begin{align*}
c+2\alpha_2 \le -1, ~~ c+\alpha_1\le -1.
\end{align*}
Similarly, since $a \ge g$ we get
\begin{align*}
c+2\alpha_1 \ge 1, ~~ c+\alpha_2 \ge 1.
\end{align*}
Altogether, we get both $c \le -3$ and $c\ge 3$. 
\qede
\end{example}

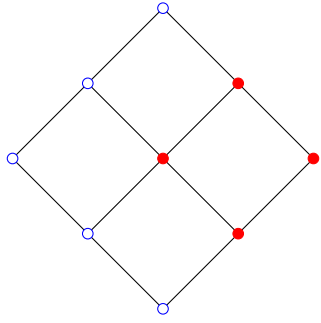
\begin{figure}[t]
\begin{center}
\begin{tikzpicture}
  \draw[-] (-2,0) -- (0,2);
  \draw[-] (0,2) -- (2,0);
  \draw[-] (2,0) -- (0,-2);
  \draw[-] (-2,0) -- (0,-2);
  \draw[-] (1,1) -- (-1,-1);
  \draw[-] (1,-1) -- (-1,1);
  \filldraw [red](2,0) circle (2pt);
  \filldraw [red](0,0) circle (2pt);
  \filldraw [red](1,1) circle (2pt);
  \filldraw [red](1,-1) circle (2pt);
  \draw[blue, fill=white] (-2,0) circle (2pt);
  \draw[blue, fill=white] (-1,-1) circle (2pt);
  \draw[blue, fill=white] (-1,1) circle (2pt);
  \draw[blue, fill=white] (0,-2) circle (2pt);
  \draw[blue, fill=white] (0,2) circle (2pt);  
\end{tikzpicture}
\end{center}
\caption{Convex separation in a lattice}\label{fig:lattice}
\end{figure}

\begin{example}
Let $(X,\wedge)$ be a semimodule induced by a semilattice $X$. This is divisible since $x\wedge x = x$. Thus, $\conv(S)$ is the sub semilattice generated by $S$. In this case convex and subadditive functions coincide, and so Theorems~\ref{thm kaufman} and~\ref{thm separate} both assert the un-obvious result that disjoint sub meet-lattices lie in partitioning sublattices. See Figure~\ref{fig:lattice}. Note that since $X$ contains nontrivial idempotent elements, it cannot be embedded in a group (see~\cite{Ham05}). See also~\cite{Pon14} for a study of convexity in semilattices.  \qede
\end{example}

\section*{Part II: Convex operators on groups}

\section{Analysis of convex operators on groups}\label{sec operators}

We turn now to results for operators on groups. By Example~\ref{ex r q} and Remark~\ref{rmk divisible}, we could derive many of these results using $\Q$-modules but we prefer to highlight the use of only monoidal structure.

\subsection{Subdifferential calculus of operators}

Here we assume that $X$, $Y$ are groups, and $f :X\to (Y\cup\{\infty\},\le)$, where $\infty$ is a maximal element with respect to the partial order $\le$ on $Y$. Assume also that $\le$ is compatible with the group operation, i.e., if $x\ge y$ iff $x-y\ge 0$. We  also assume that the order is at least \emph{inductive}, i.e., that every countable chain has an upper bound. In Subsection~\ref{subsec fenchel}, we will need to further assume that $\le$ is a \emph{complete} order, i.e., that every order bounded set has an infimum and supremum. Of course $Y$ may be $\R$ as before.

\begin{remark}A partial order in a Banach space is order complete if and only it is latticial. Moreover, order completeness of the range   characterises the Hahn-Banach extension theorem holding. By contrast if the cone has a bounded complete base, the order is inductive. Thus, in Euclidean space all pointed closed convex cones induce inductive orders.    (See \cite{BV10,Bor82,BPT84,BT92} for much more on these technicalities in the vector space setting.)\qede 
\end{remark}

As in Definition~\ref{def subadd}, $f$ is said to be \emph{subadditive} if $f(x+y) \le f(x)+f(y)$. We can similarly define $\N$-sublinear and convex functions.  
\begin{definition}[Domain of convex function]
Let $X,Y$ be groups and $f:X\to Y\cup\{\infty\}$ be convex. Define the domain of $f$ to be the set
\begin{align*}
\mathrm{dom}(f)  = \big\{x\in X ~ \big| ~ f(x)<\infty\big\}.
\end{align*}
\end{definition}

It is easily shown that the domain of a convex function of a convex subset of $X$. The \emph{core} of the domain is then:

\begin{definition}[Core of domain]
Let $X,Y$ be groups and let $f:X\to Y\cup\{\infty\}$ be a convex function. Define the core of the domain of $f$ to be 
\begin{align*}
\mathrm{core}(\mathrm{dom}(f)) = \big\{x\in X~\big|~\forall h\in X, \exists n\in \N, \exists g\in X, ng =  h, f(x+g)<\infty\big\}.
\end{align*}
\end{definition} 

By choosing $h=0$, it follows that $\mathrm{core}(\mathrm{dom}(f)) \subseteq \mathrm{dom}(f)$. More generally, we can define the core of a convex function.

\begin{definition}[Core of convex set]\label{def cor of set}
Let $X$ be a group and $C\subseteq X$ a convex set. Define the core of $C$ to be the set
\begin{align*}
\mathrm{core}(C) = \big\{x\in X~\big|~ \forall h\in X, \exists n\in \N, \exists g\in X, ng = h, x+g\in C\big\}.
\end{align*}
\end{definition}

Again, we have $\mathrm{core}(C)\subseteq C$. Now we define the directional derivative.

\begin{definition}[Directional derivative]\label{def dir der}
Let $X$ be a group, $(Y\cup\{\infty\})$ a group with an inductive order, and $f:X\to Y\cup\{\infty\}$ a convex function. For $x\in \mathrm{core}(\mathrm{dom}(f))$, define
\begin{align*}
f_x(h) = \inf\big\{ n\big(f(x+g)-f(x)\big) ~\big|~ ng = h, ~f(x+g)<\infty\big\}.
\end{align*}
\end{definition}

Before proceed to the study of directional derivatives, we need the following technical proposition.

\begin{prop}\label{prop decreasing}
Assume that $\{a_n\}_{n\in \N}$ and $\{b_n\}_{n\in \N}$ are two decreasing sequences in an inductive and compatible cone. Then
\begin{align*}
\inf_{n\in \N}\big\{a_n+b_n\big\} = \inf_{n\in \N} a_n+\inf_{n\in \N} b_b.
\end{align*}
\end{prop}

\begin{proof}
Let $n,m\in N$ with $n>m$. Then since $\{b_n\}_{n\in\N}$ is decreasing, we have $a_n+b_n \le a_n+b_m$. Thus, we have
\begin{align*}
\inf_{n\in \N}\big\{a_n+b_n\big\} \le \big[\inf_{n\in \N} a_n\big]+b_m.
\end{align*}
Taking the infimum over $m$ gives $\inf_{n\in \N}\big\{a_n+b_n\big\} \le \inf_{n\in \N} a_n+\inf_{m\in \N} b_m$. The converse inequality is clear. This completes the proof.
\end{proof}

We have the following.
\begin{prop}[One-sided derivatives, I]\label{prop derivative}
Assume that a group $X$ is a $p$-semidivisible group, and $(Y,\le)$ is a group with an inductive order. Assume also that $f:X\to Y\cup\{\infty\}$ is convex and $x\in \mathrm{core}(\mathrm{dom}(f))$. Then $f_x$ is an everywhere finite, $\N$-sublinear function.
\end{prop}

\begin{proof}
For arbitrarily large $n,n'\in \N$ with $n<n'$ we can find $g,g'\in X$ such that $ng = n'g' = h$ and $f(x+g)<\infty$, $f(x+g')<\infty$. We have $n'(x+g') = n(x+g)+(n'-n)x$, and so by convexity $n'f(x+g') \le nf(x+g)+(n'-n)f(x)$. Therefore, we have
\begin{align*}
n'(f(x+g')-f(x)) \le n(f(x+g)-f(x)).
\end{align*} 
Also, if $g,g'\in X$ are such that $ng = n'g' = h$, then $(n+n')x = n(x-g)+n'(x+g)$ and so again by convexity, we have 
\begin{align*}
n(f(x)-f(x-g)) \le n'(f(x+g')-f(x)).
\end{align*}
Thus, the sequence $\big\{n\big(f(x+g)-f(x)\big)~\big|~ ng =h,~f(x+g)<\infty\big\}$ is decreasing and bounded from below. Since $\le$ is an inductive order on $Y$, $f_x(h)$ exists and is finite. To show that $f_x(0)\le 0$, note that we can choose $g=0$ in Definition~\ref{def dir der} and obtain $f_x(0)\le 0$. To prove the positive homogeneity of $f_x$, choose, $g,g'\in X$ such that $p^lg=ph$ and $p^{l}g'=h$. Then we have $p^{l+1}(x+g') = p^{l+1}x+ph = p^{l+1}x+p^lg = (p^{l+1}-p^l)x+p^l(x+g)$. Thus, since $f$ is convex, we have
\begin{align*}
p^{l+1}f(x+g') \le (p^{l+1}-p^l)f(x)+p^lf(x+g),
\end{align*}
or in other words,
\begin{align*}
p^{l+1}\big(f(x+g')-f(x)\big) \le p^l\big(f(x+g)-f(x)\big).
\end{align*}
Taking the limit as $l\to \infty$ and using the fact that the sequence in Definition~\ref{def dir der} is decreasing, we get $pf_x(h) \le f_x(ph)$. On the other hand, we have,
\begin{eqnarray*}
pf_x(h) & = & \inf\big\{pn\big(f(x+g)-f(x)~\big|~ng= h\big\}
\\ & \stackrel{(*)}{\ge} &\inf\big\{pn\big(f(x+g)-f(x)~\big|~png=ph\big\}
\\ & \stackrel{(**)}{=} & f_x(ph).
\end{eqnarray*}
In ($*$) we used the fact that if $pg = h $ then $png=ph$ (but we might have a bigger set on which we take the infimum). In ($**$) we used the fact in Definition~\ref{def dir der} the infimum is taken over a decreasing sequence. This shows that $p_x(ph)=pf_x(h)$. Finally, to show subadditivity, note that $p(x+g_1+\dots+g_p) = (x+pg_1)+\dots+(x+pg_p)$, and so by convexity of $f$,
\begin{multline}\label{p conv}
p(f(x+g_1+\dots+g_p)-f(x))
\\ \le \big(f(x+pg_1)-f(x)\big)+\dots+\big(f(x+pg_p)-f(x)\big).
\end{multline}
Multiply~\eqref{p conv} by $n$ and then choose $g_1,\dots,g_p$ such that $ng_1 = h_1,\dots, ng_p=h_p$. This is possible since we may assume without loss of generality that $n=p^l$ for some $l\in \N$, and this is because the sequence $\big\{n\big(f(x+g)-f(x)\big)~\big|~ ng =h,~f(x+g)<\infty\big\}$ is decreasing. We get
\begin{align}\label{p conv with n}
p\big(nf(x+g_1+\dots+g_p)-f(x))\big) \le \sum_{j=1}^pn\big (f(x+g_j)-f(x)\big)
\end{align}
By Definition~\ref{def dir der}, we have
\begin{align}\label{lower bound sum}
p\big(nf(x+g_1+\dots+g_p)-f(x))\big) \ge pf_x(h_1+\dots+h_p).
\end{align}
To evaluate the right side of~\eqref{p conv with n}, note that for each $1 \le j \le p$, the sequence $$\big\{n\big(f(x+g_j)-f(x)\big)~\big|~ ng_j =h_j,~f(x+g_j)<\infty\big\}$$ is decreasing. Thus, using Proposition~\ref{prop decreasing} and taking the infimum over the right side of~\eqref{p conv with n}, we get,
\begin{align}\label{upper bound sum}
\nonumber &\inf\left\{\left.\sum_{j=1}^pn\big (f(x+g_j)-f(x)\big)~\right|~ ng_j=ph_j,~ f(x+g_j)<\infty, ~ 1 \le j \le p\right\} 
\\ \nonumber &=\sum_{j=1}^p\inf\Big\{ n\big (f(x+g_j)-f(x)\big)~\Big|~ ng_j=ph_j,~ f(x+g_j)<\infty, ~ 1 \le j \le p\Big\} 
\\  & = \sum_{j=1}^pf_x(ph_j).
\end{align}
Combining~\eqref{lower bound sum} and~\eqref{upper bound sum}, we get
\begin{align*}
pf_x(h_1+\dots h_p) \le f_x(ph_1)+\dots + f_x(ph_p),
\end{align*}
and so, since $f_x(ph_j)=pf_x(h_j)$, $1 \le j \le p$, we get
\begin{align*}
f_x(h_1+\dots h_p) \le f_x(h_1)+\dots+f_x(h_p).
\end{align*}
Note that here we used the fact that $\le$ is compatible with the group operations on $Y$, and therefore we have $py_1\le py_2 \Longrightarrow y_1 \le y_2$. Next, note that since is assumed to be $p$ prime, $p\ge 2$. Choosing $h_{3} = \dots= h_p=0$, we get
\begin{eqnarray*}
f(h_1+h_2) & \le & f_x(h_1)+f_x(h_2) +f_x(h_3)+\dots +f_x(h_p)
\\ & \stackrel{(*)}{\le} &f_x(h_1)+f_x(h_2) +0
\\ & = & f_x(h_1)+f_x(h_2).
\end{eqnarray*}
where in ($*$) we used the fact that $f_x(0)\le 0$. Altogether we have that $f_x$ is subadditive and $f(px) = pf(x)$. Now apply Proposition~\ref{prop p sufficient sublin} to deduce that $f$ is $\N$-sublinear, and the proof is complete.
\end{proof}

\begin{remark}
The proof of Proposition~\ref{prop derivative} shows that the sequence $n\big(f(x+g)-f(x)\big)$, $ng=h$, $f(x+g)<\infty$, is decreasing. If we assume that we have both $pX=X$ and $qX=X$, then in~\eqref{def dir der}, we can choose $n=p^l$ or $n=q^l$ for every $l\in \N$ and the infimum would be the same in both cases. \qede
\end{remark}

In the case when $f$ is not only convex, but actually $\N$-sublinear, we have the following stronger result.

\begin{prop}[One-sided derivatives, II]\label{prop n sublin}
Assume that $X$ is a group, $(Y,\le)$ is a group with an inductive order, and $f:X\to Y\cup\{\infty\}$ is $\N$-sublinear map, and $x\in \mathrm{core}(\mathrm{dom}(f))$. Then $f_x$ is an everywhere finite $\N$-sublinear map, that satisfies in addition $f_x(0)=0$, $f_x(x) = -f_x(-x)=  f(x)$.
\end{prop}

\begin{proof}
When $f$ is $\N$-sublinear,~\eqref{def dir der} becomes
\begin{align*}
f_x(h) = \inf\big\{f(nx+h)-nf(x)~\big|~f(nx+h)<\infty\big\}.
\end{align*}
Since $f$ is positively homogeneous, it is easy to see that $f_x(x) = -f_x(-x) = f(x)$ and $f_x(0)=0$. To show the positive homogeneity of $f_x$, use the fact that, as in the proof of Proposition~\ref{prop derivative}, the sequence $\big\{f(nx+h)-nf(x)\big\}$ is decreasing, and so we have for all $m\in \N$,
\begin{align*}
f_x(mh) & = \inf\big\{f(nx+mh)-nf(x)~\big|~f(nx+mh)<\infty\big\}
\\ & = \inf\big\{f(mkx+mh)-mkf(x)~\big|~f(mkx+mh)<\infty\big\}
\\ & = m\inf\big\{f(kx+h)-kf(x)~\big|~f(kx+h)<\infty\big\}
\\ & =mf_x(h).
\end{align*}
To show the subadditivity, take $n_1,n_2\in \N$. Since $f$ is subadditive, we have,
\begin{align*}
f_x(h_1+h_2) & \le f((n_1+n_2)x+h_1+h_2)-(n_1+n_2)f(x) 
\\ & \le \big(f(n_1x+h_1)-n_1f(x)\big)+\big(f(n_2x+h_2)-n_2f(x)\big).
\end{align*}
Taking the infimum over all $n_1,n_2\in \N$ such that $f(n_1x+h_1)<\infty$, $f(n_2x+h_2)<\infty$, the subadditivity follows.
\end{proof}

Given two monoids $X$ and $Y$, let $\mathcal L(X,Y)$ be the collection of all additive maps between $X$ and $Y$. As in the vector space setting, define the following:
\begin{align*}
\partial f(x_0) = \Big\{a\in \mathcal L(X,Y) ~\Big|~ f(x_0) + a(h) \le f(x_0+h)\Big\}.
\end{align*}
In the vector space setting it is usually required that $a(x-x_0) \le f(x)-f(x_0)$. However, in order to avoid taking differences, we use the above definition. Let $\mathcal L(X,Y)$ be the space of all additive maps between $X$ and $Y$. Then it follows that $\partial f(x_0) \subseteq \mathcal L(X,Y)$.

\begin{prop}\label{prop derivative sublin}
Assume that $X$ is a $p$-semidivisible group, $(Y,\le)$ is a group with an inductive order, and $f:X\to Y\cup\{\infty\}$ is subadditive and satisfies $f(px)=pf(x)$ for all $x\in X$. If $x\in \mathrm{core}(\mathrm{dom}(f))$, then $f_x\le f$ and
\begin{align*}
f_x(x)+f_x(-x) \le 0.
\end{align*}
\end{prop}

\begin{proof}
To prove the first assertion, note that
\begin{eqnarray*}
f_x(h) & \stackrel{(*)}{=} &\inf\big\{n\big(f(x+g)-f(x)\big)~\big|~ng =h, ~f(x+g)<\infty\big\}
\\ & = & \inf\big\{p^l\big(f(x+g)-f(x)\big)~\big|~p^lg =h, ~f(x+g)<\infty\big\}
\\ & \le & \inf\big\{p^lf(g)~\big|~p^lg =h, ~f(x+g)<\infty\big\}
\\ & \stackrel{(**)}{=} & f(h),
\end{eqnarray*}
where in ($*$) we used the fact that $\big\{n\big(f(x+g)-f(x)\big)~\big|~ ng =h,~f(x+g)<\infty\big\}$ is a decreasing sequence and in ($**$) we used the fact that $f(px) = pf(x)$. To prove the second assertion, choose $g$ such that $pg=x$ and note that 
\begin{align*}
f_x(x)+f_x(-x) & \le p\big(f(x+g)-f(x)\big)+m\big(f(x-g)-f(x)\big)
\\& = f((p+1)x)+f((p-1)x)-2pf(x)
\\& \le 0,
\end{align*}
where in the last inequality we used the subadditivity of $f$.
\end{proof}

\begin{prop}\label{prop subdif conv}
If $(Y,\le)$ satisfies that for every $m\in \N$ $my_1\le my_2 \Longrightarrow y_1 \le y_2$ then $\partial p(x_0)$ is convex in $\mathcal L(X,Y)$.
\end{prop}

\begin{proof}
For $a_1,\dots,a_n,a\in \mathcal L(X,Y)$, assume that $ma = \sum_{i=1}^nm_ia_i$, $m = \sum_{i=1}^nm_i$. Then we have
\begin{align*}
m\big(f(x_0)+a(x)\big) = \sum_{i=1}^n(f(x_0)+a_i(x)) \le \sum_{i=1}^nm_if(x)  = mf(x).
\end{align*}
By the assumption on $Y$, it follows that $f(x_0)+a(x) \le f(x)$.
\end{proof}

\subsection{The maximum or max formula}
We show that the well known max formula \cite{BV10,BL06} holds in this generality.

\begin{thm}[Max formula]\label{thm max}
Assume that $X$ is a $p$-semidivisible group, that $(Y,\le)$ is an additive group with an inductive order, and $f:X\to Y\cup\{\infty\}$ is convex. Assume also that for some $x_0\in\mathrm{core}(\mathrm{dom}(f))$, we have
\begin{align}\label{der anti symm}
f_{x_0}(x_0)+f_{x_0}(-x_0)\le 0.
\end{align}
Then we have 
\begin{align}\label{max for}
f_{x_0}(h) = \max\big\{a(h) ~ \big| ~ a \in \partial f(x_0)\big\}.
\end{align}
In particular, $f$ admits additive minorants, and $\partial f(x_0) \neq \emptyset$. The maximal element in~\eqref{max for} is bounded.
\end{thm}
 
\begin{proof}
Define $\mathcal C$ to be the set of all pairs $(\varphi,S)$, where $S\subseteq X$, and $\varphi :X\to Y\cup\{\infty\}$ is $\N$-sublinear and satisfies $\varphi\le f_{x_0}$, and $\sup_{s\in S}\big(\varphi(s)+\varphi(-s)\big)\le 0$. Define a partial order on $\mathcal C$ by 
\begin{align*}
(\varphi_1,S_1) \le (\varphi_2,S_2) \iff \varphi_1\ge \varphi_2, ~S_1 \subseteq S_2.
\end{align*}
$(\mathcal C, \le)$ is inductive, as both $\le$ and $\subseteq$ are inductive orders. By Proposition~\ref{prop derivative}, we have $f_{x_0}(0) = 0$, implying that $(f_{x_0},\{0\}) \in \mathcal C$ and so $\mathcal C\neq \emptyset$. Therefore, $\mathcal C$ has a maximal element $(\bar \varphi, \bar S)$. We claim that we must have $\bar S=X$. Otherwise, choose $y \in X\setminus \bar S$. Since $(\bar\varphi, \bar S)\in \mathcal C$, in particular it follows that the function $\bar \varphi$ satisfies the hypotheses of Proposition~\ref{prop derivative sublin}. Also, $y\in \mathrm{core}(\mathrm{dom}(f))$ since $\bar \varphi \le f_{x_0}$ and $f_{x_0}$ is everywhere finite (by Proposition~\ref{prop derivative}). Therefore, Proposition~\ref{prop derivative sublin} implies that $\bar\varphi_y \le \bar \varphi$ and $\bar \varphi_{y}(y)+\bar \varphi_{y}(-y)\le 0$. This means that $(\bar \varphi_{y},\bar S\cup \{y\})\in \mathcal C$, which is a contradiction to the maximality of $(\bar \varphi, \bar S)$. Thus, we have $\bar S=X$. Next, we claim that $\bar \varphi$ is additive on $X$. If not, then since $\bar \varphi$ is subadditive, there must exist $x,h\in X$ such that $\bar \varphi(x+h)-\bar \varphi(h)<\bar \varphi(x)$. But then $\bar \varphi_{x} \le \bar \varphi$ which is again a contradiction to the maximality of $(\bar \varphi, \bar S)$. Since $\bar \varphi\le f_{x_0}$ and by~\eqref{der anti symm} we have $\bar \varphi(-x_0)\le f_{x_0}(-x_0) \le-f_{x_0}(x_0)$, it follows that $\bar \varphi(x_0)=f_{x_0}(x_0)$ and $\bar\varphi$ is bounded. Choosing $a=\bar \varphi$ proves~\eqref{max for}. Since $x\in \mathrm{core}(\mathrm{dom}(f))$, Definition~\ref{def dir der} implies that the maximal element in~\eqref{max for} is indeed bounded. This completes the proof.
\end{proof}

An instructive setting is when $Y$ is the symmetric matrices endowed with the (non-lattical)  semidefinite order.

\begin{remark}[Well posedness]
If $f$ is $\N$-sublinear and $ng = x$, then by positive homogeneity, we have $n\big(f(x-g)-f(x)\big) = f((n-1)x)-f(nx) = -f(x)$ and $n\big(f(x+g)-f(x)\big) = f((n+1)x)-f(nx) = f(x)$. In particular, $f_x(x)+f_x(-x) \le 0$ for every $x\in \mathrm{core}(\dom(f))$. Thus, every $\N$-sublinear function satisfies the assumptions of Theorem~\ref{thm max}. \qede
\end{remark}

\begin{remark}
Using Proposition~\ref{prop n sublin}, we have that Theorem~\ref{thm max} holds if $f$ is $\N$-sublinear, even if we omit the subdivisibility assumption. \qede
\end{remark}

\subsection{Fenchel-Rockafellar duality}\label{subsec fenchel}
As in vector spaces, define the \emph{additive dual group} of a group $X$ to be 
\begin{align*}
X^* = \big\{\varphi:X\to \R ~\big|~ \varphi \text{ is additive}\big\}.
\end{align*}
Then 
$X^*$ is an additive group with the addition being  point-wise addition. We emphasise that $X^*$ is \emph{not} the group of homomorphisms of $X$. How rich a notion this is depends on the given group.

Consider now $(Z, \le)$ which is order complete. We still require that $\le$ is compatible with the group operation. Define the \emph{conjugate function} $f^{\star}:X^*\to Z\cup\{\infty\}$ to be
\begin{align}\label{def:fenchel}
f^{\star}(\varphi) = \sup_{x\in X}\big\{\varphi(x)-f(x)\big\}.
\end{align}
The conjugate function has been studied extensively in the vector space setting. See for example~\cite{BL06, BV10, Rock97}. Note that $f^*(\varphi)=+\infty$ will happen if \eqref{def:fenchel} has no upper bound. Before proving the \emph{Fenchel duality theorem} for groups, we need the following proposition.

\begin{prop}\label{prop inf conv}
Assume that $X_1,X_2,Z$ are groups, where $X_1$ is semidivisible and $(Z,\le)$ is an order complete group. Let $T:X_1\to X_2$ be additive, and assume that $f:X_1\to Z\cup\{\infty\}$ and $g:X_2\to Z\cup\{\infty\}$ are convex. If we define $h:X_2\to Z\cup\{\infty\}$ by
\begin{align*}
h(u) = \inf_{x\in X_1}\big[f(x)+g(Tx+u)\big],
\end{align*}
then $h$ is convex, and it domain is given by
\begin{align}\label{dom h}
\dom(h) = \dom(g)-T\dom(g).
\end{align}
\end{prop}

\begin{proof}
First, note that since $g$ is convex and $T$ is additive, it follows that $g\circ T:X_1\to Z\cup\{\infty\}$ is convex. Next, to show the convexity of $h$, let $m_1,\dots,m_n\in \N$, $u_1,\dots,u_n,u\in X_2$ such that $mu = \sum_{i=1}^nm_iu_i$, $m=\sum_{i=1}^nm_i$. Let $x_1,\dots,x_n\in X_1$. By Proposition~\ref{prop p sufficient}, we may assume that $m = p^l$, where $p$ is a prime satisfying $pX=X$. Hence, there exists $x\in X_1$ such that $mx = \sum_{i=1}^nm_ix_i$. We have
\begin{align*}
mh(u) & \le m\big(f(x)+g(Tx+u)\big) 
\\ & \le \sum_{i=1}^nm_i\big(f(x_i)+g(Tx_i+u_i)\big).
\end{align*}
Taking the infimum over $x_1,\dots,x_n\in X$, we get 
\[mh(u) \le \sum_{i=1}^nm_ih(u_i).\]
The proof of~\eqref{dom h} is immediate. This completes the proof. 
\end{proof}

\begin{thm}[Fenchel-Young inequality for groups]\label{fenchel-young}
Suppose that $X$, $Z$, are groups, $Z$ is order complete, and $f:X\to Z\cup\{\infty\}$. Then for every $x\in X$ and every $\varphi \in X^*$,
\begin{align*}
f(x)+f^{\star}(\varphi) \ge \varphi(x).
\end{align*}
Equality holds if and only if $\varphi\in \partial f(x)$.
\end{thm}

\begin{proof}
By definition~\eqref{def:fenchel}, $\varphi(x) - f(x) \le f^{\star}(\varphi)$ which implies $f(x)+f^{\star}(\varphi) \ge \varphi(x)$. If $\varphi\in \partial f(x)$, then $f(x)+\varphi(y-x) \le f(y)$ and so $f(x)-\varphi(x) \le f(y)-\varphi(y)$. Taking the infimum over the right side gives $f(x)-\varphi(x) \le -f^{\star}(\varphi)$ which then gives $f(x) + f^{\star}(\varphi) = \varphi(x)$. Conversely, by the definition of $f^{\star}$, if $f(x)+f^{\star}(\varphi) = \varphi(x)$ then $\varphi(y-x) \le f(y)-f(x)$, and so $\varphi\in \partial f(x)$ as required.
\end{proof}

\begin{example}
If $X$ is a meet lattice then additive functions are identically 0, since for every $m\in \N$ we have 
\[f(x) = f(\overbrace{x\wedge \dots \wedge x}^{m\text{ times}}) = mf(x).\] 
Hence $X^*=\{0\}$ and Theorem~\ref{fenchel-young} simply gives $f(x) \ge \inf_{x\in X}f(x)$.\qede
\end{example}

For an additive map $T: X_1\to X_2$ define the \emph{adjoint} $T^*:X_2^*\to X_1^*$ in the usual way
\begin{align*}
(T^*x_2^*)(x_1) = x_2^*(Tx_1), ~~~ x_1 \in X_1, ~ x_2^*\in X_2^*.
\end{align*}

We are now in a position to state and prove the Fenchel duality theorem.

\begin{thm}[Weak and strong Fenchel duality]\label{thm fenchel}
Let $X_1,X_2, Z$, be groups, and $(Z,\le)$ an order complete group. Given $f: X_1\to Z\cup \{\infty\}$, $g:X_2\to Z\cup\{\infty\}$ and an additive map $T:X_1\to X_2$, define
\begin{align*}
& P = \inf_{x\in X_1}\big\{f(x)+g(Tx)\big\},
\\
& D  = \sup_{\varphi^*\in X_2^*}\big\{-f^{\star}(T^*\varphi)-g^{\star}(\varphi)\big\}.
\end{align*} 
Then $P \ge D$ (weak duality). In particular, if $P=-\infty$ then $D=-\infty$. If, in addition,  $X_1$ is semidivisible, $f$ and $g$ are convex and we assume
\[0 \in \mathrm{core}\big(\mathrm{dom}(g)-T\,\mathrm{dom}(f)\big),\]
then $P = D$ (strong duality) and $D$ is attained when finite.
\end{thm}

\begin{proof}
To prove weak duality, note that $P\ge D$ is equivalent to
\begin{align*}
\inf_{\substack{x\in X_1 \\ \varphi \in X_2^*}}\Big[f(x)+f^{\star}(T^*\varphi)+g(Tx)+g^{\star}(-\varphi)\Big] \ge 0.
\end{align*}
By Theorem~\ref{fenchel-young}, we have $f(x)+f^{\star}(T^*\varphi) \ge (T^*\varphi)(x)$ and $g(Tx)+g^{\star}(-\varphi) \ge -\varphi(Tx)$. 
Then by the definition of $T^*$ we have $(T^*\varphi)(x)-\varphi(Tx) = 0$. 

To prove strong duality, define $h:X_2 \to Z\cup \{\infty\}$,
\begin{align*}
h(u) = \inf_{x\in X_1}\big\{f(x)+g(Tx+u)\big\}.
\end{align*}
By Proposition~\ref{prop inf conv}, $h$ is convex and $\dom(h) = \dom(g)-T\,\dom(f)$ is a convex set. Since we assume that $0\in \core\big(\dom(g)-T\,\dom(f)\big)$, applying Theorem~\ref{thm max} for $h$ and $x_0=0$ implies that there exists $\varphi:X_2\to Z\cup\{\infty\}$ additive such that $\varphi(u) \le h(u)-h(0)$ (note that since we choose $x_0=0$ in Theorem~\ref{thm max}, the condition $h_{x_0}(x_0)+h_{x_0}(-x_0) \le 0$ holds, as $h_{x}(0)=0$ always). Hence,
\begin{align*}
h(0) & \le h(u) -\varphi(u) \le f(x)+g(Tx+u) -\varphi(u) 
\\ & = \big[f(x)-(T^*\varphi)(x)\big] + \big[g(Tx+u) - (- \varphi(Tx+u))\big].
\end{align*}
Taking the infimum over $x\in X_1$, $u\in X_2$ implies
\begin{align*}
h(0) \le -f^{\star}(T^*\varphi) - g^{\star}(-\varphi) \le D.
\end{align*}
Since $h(0) = P$, strong duality follows. Again the dual supremum is attained when finite.
\end{proof}

\begin{example}
If $X_2$ is a meet lattice, then $X_2^*=\{0\}$ and 
\begin{align*}
D = -f^{\star}(0)-g^{\star}(0) = \inf_{x\in X_1}f(x)+\inf_{x\in X_2}g(x)
\end{align*}
which is clearly smaller than $P$. \qede
\end{example}

\begin{remark}\label{fenchel sublin}
Assume that in Theorem~\ref{thm fenchel} we have $\N$-sublinear functions rather than convex functions. Then if we use Proposition~\ref{prop n sublin}, Theorem~\ref{thm fenchel} still holds even if we omit the subdivisibility assumption. 
\qede
\end{remark}

Next we discuss  applications of Theorem~\ref{thm fenchel}. One of the classical applications, is a representation for the subdifferential of a sum of convex functions. We show that such a result holds for groups as well.

\begin{thm}[Sum rule for subdifferentials]\label{thm sum rule}
Suppose  $f:X_1\to Z\cup\{\infty\}$, $g:X_2\to Z\cup\{\infty\}$, for $(Z,\le)$  an order complete group and $T:X_1\to X_2$  is additive. Then
\begin{align*}
\partial\big(f+g\circ T\big)(x_0) \supseteq \partial f(x_0) + T^*\partial g(x_0).
\end{align*}
If, in addition, $X_1$ is semidivisible,   $0 \in \mathrm{core}\big(\mathrm{dom}(g)-T\,\mathrm{dom}(f)\big)$,while  $f$ and $g$ are convex, then equality holds.
\end{thm}

\begin{proof}
The first inclusion follows immediately. To prove the equality case, let $\phi\in \partial\big(f+g\circ T\big)(x_0)$. Then the function $(f-\phi) + g\circ T$ is minimised at $x_0$. Assume without loss of generality that the minimum is 0. By the strong Fenchel duality result with $P=D=0$, there exists $\varphi\in X_2^*$ such that
\begin{align*}
0 = -(f-\phi)^{\star}(T^*\varphi)-g^{\star}(-\varphi) = -f^{\star}(T^*\varphi+\phi)-g^{\star}(-\varphi).
\end{align*}
Hence, for every $x_1\in X_1$ and $x_2\in X_2$, we have
\begin{align}\label{sep two vars}
0 \le (f-\phi)(x_1)-T^*\varphi(x_1)+g(x_2)+\varphi(x_2).
\end{align}
In particular, choosing $x_1=x_0$, we have for all $x_2\in X_2$,
\begin{align*}
-\varphi(x-Tx_0) \le (f-\phi)(x_0)+g(x_2) = g(x_2)-g(Tx_0),
\end{align*}
where in the last equality we used our assumption that $(f-\phi)(x_0)+g(Tx_0) = 0$. Thus, we have $-\varphi\in \partial g(Tx_0)$. Also, by~\eqref{sep two vars}, we have
\begin{align*}
\sup_{x_1\in X_1}\big(-g(Tx_1)-T^*\varphi(x_1)\big) \le \inf_{x_1\in X_1}\big((f-\phi)(x_1)-T^*\varphi(x_1)\big).
\end{align*}
Thus there exists $z_0 \in Z$ such that for all $x_1\in X_1$,
\[-g(Tx_1) \le (T^*\varphi)(x_1) +z_0 \le (f-\phi)(x_1)\]	
and equality holds when $x_1=x_0$. Hence $z_0=0$ and $T^*\varphi+\phi\in \partial f(x_0)$, which completes the proof of the theorem.
\end{proof}

Another application of Theorem~\ref{thm fenchel} is a Hahn-Banach theorem for groups.

\begin{thm}[Hahn-Banach theorem for groups]
Let $X$ be a group, $X'\subseteq X$ a subgroup, and $(Z,\le)$ an order complete group. Assume that $f:X\to Z$ is $\N$-sublinear and $h:X'\to Z$ is additive such that $h\le f$ on $X'$. Then there exists $\bar h:X\to Z$ additive such that $\bar h \le f$ and $\bar h = h$ on $X'$.
\end{thm}

\begin{proof}
Choose $X_1=X_2=X$ and let $T:X\to X$ be the identity map.  Choose $g:X'\to Z\cup\{\infty\}$ to be $g = -h+\iota_{X'}$, where
\[\iota_{X'}(x) = \begin{cases} 0 & x\in X', \\ \infty & x\notin X'.\end{cases}\]
Since $f:X\to Z$, $\dom(f)=X$. Also, $\dom(g) = X'$. Thus $0\in \mathrm{core}(\dom(f)-T\dom(g))$ and we can thus use Theorem~\ref{thm fenchel}. Note that by Remark~\ref{fenchel sublin} we do not need to assume subdivisibility as we are dealing with $\N$-sublinear functions. Now, by Theorem~\ref{thm fenchel}, we have
\begin{align}\label{bigger than zero}
\nonumber 0 & \le \inf_{x\in X}\big\{f(x)-h(x)+\iota_{X'}(x)\big\} 
\\ \nonumber & = \inf_{x\in X}\big\{f(x)+g(x)\big\} 
\\ & = \sup_{\varphi\in X^*}\big\{-f^{\star}(\varphi)-g^{\star}(-\varphi)\big\}.
\end{align}
Thus, there exists $\varphi\in X^*$ such that for all $x\in X'$, $f^{\star}(\varphi) \le \varphi(x) - h(x)$. Since $f$ is sublinear, $f(0)=0$ and so it follows that $f^{\star}(\varphi)\ge 0$ or in other words $h(x) \le \varphi(x)$, $x\in X'$. Since $X'$ is a subgroup and $\varphi$ is additive, we have $h(x) = \varphi(x)$ on $X'$ and $g^{\star}(-\varphi)=0$. Now~\eqref{bigger than zero} implies that $f^{\star}(\varphi) =0$, which implies that $\varphi(x) \le f(x)$ for all $x\in X$.
\end{proof}

\begin{remark}
If $X$ and $Z$ are groups and $f,g:X\to Z\cup\{\infty\}$ are additive with $g\le f$, then $f=g$. However, if $X$ is  only a semigroup, this is no longer always true. As a result, we cannot expect strong Hahn-Banach type theorems on arbitrary semigroups. \qede
\end{remark}

\begin{thm}[Sandwich theorem for groups]\label{thm sandwich}
Assume that $X_1$ is a semidivisible group, $X_2$ a group, and $(Z\cup\{\infty\},\le)$ a group with complete order. Let $f:X_1\to Z\cup\{\infty\}$, $-g:X_2\to Z\cup\{\infty\}$ be convex and $T:X_1\to X_2$ be additive, such that $g\circ T\le f$. Assume that $0\in \mathrm{core}(\mathrm{dom}(g)-T\,\mathrm{dom}(f))$. Then there exists an additive function $a:X\to Z$ such that $g\circ T \le a \le f$.
\end{thm}

\begin{proof}
Using Theorem~\ref{thm fenchel}, we have $P\le 0$ and so there exists $\varphi\in X_2^*$ such that $ (-g)^{\star}(-\varphi) \le -f^{\star}(T^*\varphi)$. This implies that
\begin{align}\label{ineq sup inf}
\sup_{x\in X_1}\big(-g(Tx)-T^*\varphi(x)\big) \le \inf_{x\in X_1}\big(f(x)-T^*\varphi(x)\big).
\end{align}
$T^*\varphi$ is the required additive function. If $P=-\infty$ in Theorem~\ref{thm fenchel}, then $P<-\alpha<0$ for every $\alpha>0$ and so inequality~\eqref{ineq sup inf} still holds. 
\end{proof}

\begin{remark}
By Proposition~\ref{prop n sublin},  Theorem~\ref{thm sandwich} holds if we replace convex functions by $\N$-sublinear, even if we omit the subdivisibility assumption. \qede
\end{remark}

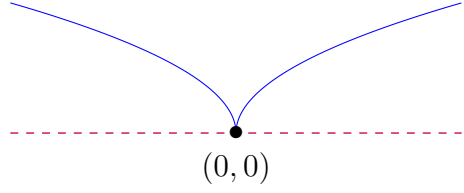
\begin{figure}[t]
\begin{center}
\begin{tikzpicture}
  \draw [draw=blue] plot[smooth, samples=1000, domain=-3:3] ({\x},{sqrt(abs(\x))});
  \draw[line width=0.5pt, dashed, scale=1,domain=-3:3,smooth,variable=\x, purple] plot ({\x},{0});
  \node at (0,0) {\textbullet};
  \node [below=3pt] at (0,0) {$(0,0)$};
\end{tikzpicture}
  \caption{Single minorant in $\R$}\label{fig:minorant}
\end{center}
\end{figure} 
 
\begin{remark}
Even for $X=\R$, the only additive minorant may be $a=0$. Consider  the subadditive (non-convex) function $f(x) = \sqrt{|x|}$. See Figure~\ref{fig:minorant}. \qede
\end{remark}

\section{Subadditive optimisation}\label{sec optimize}
 
Let $f,g_1,\dots,g_k:X\to [-\infty,\infty]$ and $b\in \R$. Define $v:\R^k\to[-\infty,\infty]$ by
\begin{align}\label{def v}
v(b) = v(b_1,\dots,b_k) = \inf\big\{f(x)~\big|~ x\in X, ~g_1(x)\le b_1,\dots,g_k(x) \le b_k\big\}.
\end{align}
$v$ is also known as the \emph{value function}. We have the following.
\begin{prop}[Subadditive and sublinear value functions]\label{prop:value}
~Assume that $X$ is a monoid and $f,g_1,\dots,g_k:X\to [-\infty,\infty]$ are subadditive. Then the function $v:\R^k\to [-\infty,\infty]$ defined by~\eqref{def v} is subadditive. If, in addition, $X$ is assumed to be $p$-semidivisible and $f,g_1,\dots,g_k$ satisfy $f(px)=pf(x)$, $g_i(px)=pg_i(x)$, $1 \le i \le k$ then $v$ satisfies $v(px)=pv(x)$. In particular, by Proposition~\ref{prop p sufficient sublin}, $v$ is convex.
\end{prop}

\begin{proof}
Let $x_1,x_2\in X$ be such that $g_i(x_1)\le b_i$, $g(x_2)\le c_i$, $1\le i \le k$. Since $g_1,\dots,g_k$ are subadditive, $g_i(x_1+x_2)\le g_i(x_1)+g_i(x_2)\le b_i+c_i$, $1\le i \le k$. Thus, of $b=(b_1,\dots,b_k)$, $c=(c_1,\dots,c_k)$, then
\begin{align*}
v(b+c) \le f(x_1+x_2) \le f(x_1)+f(x_2),
\end{align*}
where we used the subadditivity of $f$. Taking the infimum over the right side, the first assertion follows. To prove the second assertion, we only need to prove positive homogeneity. Indeed, for every $x\in X$, since $X$ is $p$-semidivisible, there exists $y\in X$ satisfying $x=py$. As a result,
\begin{align*}
v(pb) & = \inf\big\{f(x)~\big| ~ x\in X, ~g_1,(x)\le pb_1,\dots,g_k(x)\le pb_k\big\}
\\ & = \inf\big\{f(py)~\big| ~ y\in X,~  g_1(py)\le p b_1,\dots,g_k(py)\le pb_k\big\}
\\ & = \inf\big\{pf(y)~\big| ~ y\in X,~  pg_1(y)\le p b_1,\dots,pg_k(y) \le pb_k\big\}
\\ & = p\,\inf\big\{f(y)~\big| ~ y\in X,~  g_1(y)\le b_1,\dots,g_k(y) \le b_k\big\}
\\ & = pv(b),
\end{align*}
and we are done.
\end{proof}

\begin{remark}
The result holds if the module is over a semidivisible semiring $R$ and $f$ and $g$ are subadditive functions. \qede
\end{remark}

In the sublinear case, we may now apply Theorem \ref{thm max} to the function $h$ of Proposition \ref{prop:value} to describe $h$ in terms of additive minorants.

\begin{example}Let $b\in \R$, and let
\begin{align*}
\inf\big\{-x ~ \big| ~ 2x\le b, ~x\in \Z\big\} = -\left\lceil\frac b 2 \right\rceil.
\end{align*}
Thus, in the nondivisible setting, even if $k=1$ and $f$ and $g_1$ are additive, $v$ need not be homogeneous. \qede
\end{example}

In general integer programming \cite{Wil97,AV95} adding the sub additive, but not $\N$-homogeneous,  \emph{ceiling} function $\lceil\,  \cdot\,  \rceil$  allows one to reconstruct integer value functions but the additive minorants do not suffice. This is discussed in~\cite{TW81, BJ82}. It is interesting to ask what class of groups allows an analogue of the ceiling?

We note also that methods that were originally developed to study linear programming results in vector spaces, such as the cutting-plane method~\cite{Kel60}, can also be used to study \emph{integer} linear programming problems. See also~\cite{AV95, LL02} and the survey~\cite{BV03} for more information on the cutting-planes method, and~\cite{BJ82, Gom58, GB60, LL02} for more information on integer programming.

\subsection{Lagrange multipliers in action}

Suppose now that we have an optimisation problem with $m$ constraints:
\begin{align*}
\inf\big\{f(x)~\big|~g_1(x) \le 0,\dots,g_k(x)\le 0\big\}.
\end{align*}
Let $g(x) = (g_1(x),\dots,g_k(x))\in \R^m$. Define the \emph{Lagrangian function} $L:X\times \R^k \to (-\infty,\infty]$ to be 
\begin{align*}
L(x,\lambda) = f(x)+\lambda\cdot g(x).
\end{align*}
Here, $\lambda \cdot g(x)$ is the standard inner product in $\R^k$. We say that $\bar\lambda \in \R^k$ is a \emph{Lagrange multiplier} if the Lagrangian function $L(\,\cdot \, , \bar\lambda)$ has the same infimum as $f$ on $X$.
We will now show that Lagrange multipliers can be used to compute the subdifferential of the maximum of convex function. In the vector space case, this fact has several different proofs. We chose this particular version to show the use of Lagrange multipliers in the group setting.

\begin{thm}\label{thm subdif max}
Let $X$ be a semidivisible group and $f_i:X\to (-\infty,\infty]$ be convex functions, where $i\in I$, $I$ being a finite index set. Let $f = \max_{1\le i \le k}f_i$. For $x_0\in \bigcap_{i\in I(x_0)}\mathrm{core}(\mathrm{dom}(f_i))$, where  $I(x_0) = \{ 1\le i \le k ~|~ f_i(x_0) = f(x_0)\}$. Then we have
\begin{align*}
\partial f(x_0) = \conv\Bigg(\bigcup_{i\in I(x_0)}\partial f_i(x_0)\Bigg).
\end{align*}
\end{thm}

\begin{proof}
The inclusion $\supseteq$ follows immediately from the fact the subdifferential is convex (Proposition~\ref{prop subdif conv} with $Y=\R$). To prove the other inclusion, consider the constrained minimisation problem
\begin{align}\label{max as inf}
\inf\big\{ t ~ \big|~ t\in \R,~x\in X,~ f_1(x) \le t ,\dots,f_k(x) \le t\big\}.
\end{align}
Note that this infimum equals $\inf_{x\in X}f(x)$. Assume first that $0\in \partial f(x_0)$, which means that the infimum in~\eqref{max as inf} is attained at $x_0$. Define the following auxiliary value function $v:\R^{I(x_0)}\to [-\infty,\infty]$,
\begin{align*}
v(b) = \inf\big\{t ~\big|~ f_i(x)-t\le b_i, ~~i \in I(x_0)\big\}.
\end{align*}
We have $v(b) \ge f(x_0)-\max_{i\in I(x_0)}|b_i|>-\infty$. Also, since we assumed that 
\[x_0 \in\bigcap_{i\in I(x_0)}\mathrm{core}(\mathrm{dom}(f_i)),\] 
it follows that $0\in \mathrm{core}(\dom(v))$. By Proposition~\ref{prop:value}, $v$ is convex. Thus, by Theorem~\ref{thm max}, there exists $\bar \lambda\in \partial v(0)$ (again we are allowed to use the max formula because we are at $x_0=0$). We note also that if $b\in \R_+^{I(x_0)}$ then we also have $v(b) \le f(x_0)$ (infimum over a larger set) and also $v(0) = f(x_0)$. Thus, we have
\begin{align*}
f(x_0) = v(0) \le v(b)+\bar \lambda\cdot b \le f(x_0)+\bar\lambda\cdot b,
\end{align*}
which means that $\bar\lambda \in \R_+^{I(x_0)}$. Hence, 
\begin{align*}
t & \ge v((f_i(x)-t)_{i\in I(x_0)}) 
\\  & \ge v(0) - \bar\lambda\cdot (f_i(x)-t)_{i\in I(x_0)}  
\\ & = f(x_0) -\bar\lambda\cdot (f_i(x)-t)_{i\in I(x_0)},
\end{align*}
and so
\begin{align*}
t+\bar\lambda\cdot (f_i(x)-t)_{i\in I(x_0)} \ge f(x_0), 
\end{align*} 
which means that $\bar\lambda$ is a minimiser for the Lagrangian function. In other words, we can find $\bar\lambda\in \R^{I(x_0)}$ that minimises
\begin{align}\label{expression min}
t+\sum_{i\in I(x_0)}\lambda_i(f_i(x)-t) = t\left(1-\sum_{i\in I(x_0)}\lambda_i\right)+\sum_{i\in I(x_0)}\lambda_if_i(x).
\end{align} 
We must have $\sum_{i\in I(x_0)}\bar\lambda_i = 1$. If not, then we can choose $t$ that would make~\eqref{expression min} go to $-\infty$. Thus, we have 
\begin{align*}
\sum_{i\in I(x_0)}\bar\lambda_if_i(x_0) \le \sum_{i\in I(x_0)}\bar\lambda_if_i(x),
\end{align*}
and so $0\in \partial \left(\sum_{i\in I(x_0)}\bar\lambda_if_i\right)(x_0)$. If, in general, we have that $\phi \in \partial f(x_0)$, then $0\in \partial (f-\phi)(x_0)$ and then we repeat the same argument to conclude that $\phi \in \partial \left(\sum_{i\in I(x_0)}\bar\lambda_if_i\right)(x_0)$. Altogether, we get 
\begin{align*}
\partial f(x_0) \subseteq \bigcup\left\{\partial\Big(\sum_{i\in I(x_0)}\lambda_if_i\Big)(x_0) ~ \Bigg|~ \lambda_i\ge 0, \sum_{i\in I(x_0)}\lambda_i=1\right\}.
\end{align*}
Now, Theorem~\ref{thm sum rule} implies that the right side is equal to
\begin{align*}
\conv\Bigg(\bigcup_{i\in I(x_0)}\partial f_i(x_0)\Bigg),
\end{align*}
and so we have
\begin{align*}
\partial f(x_0)\subseteq \conv\Bigg(\bigcup_{i\in I(x_0)}\partial f_i(x_0)\Bigg),
\end{align*}
which proves the other inclusion and concludes the proof.
\end{proof}

\begin{remark}
Combining Theorem~\ref{thm subdif max} with Proposition~\ref{prop still conv} allows us to consider subadditive optimisation problems with finitely many constraints.
\end{remark}

 \section{Conclusion}\label{sec conc}
 This paper grew out of a lecture that the first author gave in 1983 and then put aside until 2015 when the second author joined him in recreating and extending the original results. One original intention was to better understand the difficulty of integer programming as  that of programming over a non-divisible group. See also~\cite{BEEFST14, FGL05}.  In so doing we have uncovered many interesting connections but as of now made little progress directly for integer programming.
 
Surely there are many other classical results for which one can find elegant and even useful generalisations. Hopefully this paper will serve as an invitation to others to join the  pursuit.


\end{document}